\documentclass[11pt]{article}
\usepackage{amsmath}
\usepackage{amsfonts}
\usepackage{latexsym}
\usepackage{amssymb}
\usepackage{graphicx,psfrag}
\usepackage{amsthm}
\usepackage{fullpage}
\usepackage{framed}
\usepackage{verbatim}
\usepackage{color}
\usepackage{epsfig}
\usepackage{authblk}
\usepackage{hyperref}
\usepackage{indentfirst}
\usepackage{mathtools}
\usepackage{float}
\usepackage{enumerate}
\usepackage[margin=1 in]{geometry}

\theoremstyle{plain}
\newtheorem{theorem}{Theorem}[section]
\newtheorem{corollary}[theorem]{Corollary}
\newtheorem{proposition}[theorem]{Proposition}

\newtheorem{lemma}[theorem]{Lemma}

\theoremstyle{definition}
\newtheorem{definition}[theorem]{Definition}

\newtheorem{remark}[theorem]{Remark}

\theoremstyle{remark}

\newcommand\op{\operatorname}

\newcommand\R{\mathbb{R}}

\newcommand\x{\boldsymbol{x}}
\newcommand\w{\boldsymbol{w}}
\newcommand\s{\boldsymbol{s}}
\newcommand\bk{\boldsymbol{k}}

\newcommand\f{\boldsymbol{f}}
\newcommand\by{\boldsymbol{y}}

\def\GG{\mathcal{G}}
\newcommand{\supp}{\op{supp}}

        {\begin{list}
                {\noindent\makebox[0mm][r]{$\bullet$}}
                {\leftmargin=5.5ex \usecounter{enumi}

      \topsep=1.5mm \itemsep=-.75ex}
        }
        {\end{list}}

\usepackage{tikz}
\usetikzlibrary{arrows.meta, positioning}

\begin{document}

\title{Autocatalytic systems and recombination: a reaction network perspective}

\author[1]{Gheorghe Craciun}
\author[2]{Abhishek Deshpande}
\author[3]{Badal Joshi}
\author[4]{Polly Y. Yu}
\affil[1]{Department of Mathematics and Department of Biomolecular Chemistry, University of Wisconsin-Madison, {\tt craciun@math.wisc.edu}.}
\affil[2]{Department of Mathematics, University of Wisconsin-Madison, {\tt deshpande8@wisc.edu}.}
\affil[3]{Department of Mathematics, California State University San Marcos, {\tt bjoshi@csusm.edu}.}
\affil[4]{Department of Mathematics, University of Wisconsin-Madison, {\tt pollyyu@math.wisc.edu}.}

\maketitle
\begin{abstract}
Autocatalytic systems are very often incorporated in the ``origin of life" models, a connection that has been analyzed in the context of the classical hypercycles introduced by Manfred Eigen. We investigate the dynamics of certain networks called \emph{bimolecular autocatalytic systems}. In particular, we consider the dynamics corresponding to the {\em relative populations} in these networks, and show that they can be analyzed by studying  well-chosen  autonomous 
polynomial dynamical systems. Moreover, we find that one can use results from {\em reaction network theory} to prove persistence and permanence of several types of bimolecular autocatalytic systems called \emph{autocatalytic recombination networks}.
\end{abstract}

\section{Introduction}

Biological networks display a wide range of sophisticated dynamical behaviours. A particular instance of this  is manifested in autocatalytic reaction networks called {\em hypercycles}, which have appeared in the ``origin of life" models. Introduced by Manfred Eigen~\cite{eigen1971selforganization} and extended in collaboration with Peter Schuster~\cite{eigen1977principle,eigen1982stages},  hypercycles consist of cyclic connections of molecules that replicate each other by mutual catalysis. The hypercycles are particular instances of the replicator equation~\cite{schuster1983replicator}, and have been proposed as a possible solution to Eigen's paradox~\cite{eigen1971selforganization}, a famous problem in evolutionary biology. The classical three dimensional hypercycle is given by the following network: $\{X_1 + X_2 \rightarrow X_1 + 2X_2, X_2 + X_3 \rightarrow X_2 + 2X_3, X_3 + X_1 \rightarrow X_3 + 2X_1\}$. The dynamical equations corresponding to the classical hypercycle are given by
\begin{eqnarray}\label{eq:hypercycle}
\begin{split}
\dot{x_1} &= x_1x_3\\
\dot{x_2} &= x_1x_2\\
\dot{x_3} &= x_2x_3
\end{split}
\end{eqnarray}
As is evident from Equation~(\ref{eq:hypercycle}), the rate equations corresponding to each species consist of only positive terms. Consequently, the species populations can become unbounded in finite time~\cite{joshi2020autocatalytic}. Therefore, we analyze the {\em relative populations} of species in such networks. Taking cue from the classical hypercycle, we consider generalized models called \emph{bimolecular autocatalytic systems}, where every reaction is of the form $X_i + X_j \rightarrow X_i + X_j +X_k$, where $i,j,k\in\{1,2,3,...,n\}$. A priori, the dynamics of {\em relative} populations may {\em not} be given by autonomous polynomial differential equations, even if that of the absolute populations is. Nevertheless, we show that the relative populations of species in these networks are, after time re-scaling, solutions of autonomous 
polynomial dynamical systems. 

Hofbauer, Schuster, Sigmund and Wolff  ~\cite{hofbauer1981general,schuster1978dynamical,hofbauer1980dynamical,schuster1979dynamical,hofbauer1984difference}  analyzed a dynamical property called \textit{cooperation}, which roughly implies that no species can go extinct. This is related with the property of \textit{persistence} in the theory of reaction networks~\cite{craciun2013persistence}, which means that given an initial condition $\x(0)\in\mathbb{R}^n_{>0}$, we have $\displaystyle\liminf_{t\rightarrow\infty} x_i(t) >0$ for all species concentrations $x_i$ in the network. We show that results from the theory of reaction networks~\cite{ craciun2013persistence, pantea2012persistence, anderson2011proof, gopalkrishnan2014geometric} can be used to prove cooperation in several kinds of \emph{autocatalytic recombination networks}. 

The paper is organized as follows. In Section~\ref{sec:reaction_networks} we introduce some definitions and notations from the theory of reaction networks. Reaction networks can be regarded as graphs embedded in Euclidean space, called Euclidean embedded graphs (E-graphs). More specifically, we define {\em weakly reversible}, {\em endotactic}, and {\em strongly endotactic} reaction networks. Further, we define the notions of {\em persistence} and {\em permanence} and relate them to the \emph{Global Attractor  Conjecture}~\cite{craciun2009toric}. In Section~\ref{sec:autocatalysis} we analyze the dynamics of certain bimolecular autocatalytic systems. In Theorem~\ref{thm:dynamics_relative_concentrations} we show that the relative populations in these networks can be obtained as time re-scaled solutions of polynomial mass-action systems. In Theorem~\ref{thm:relative_autocatalysis} we explicitly characterize the reaction networks that generate the dynamics of relative populations in such networks. Finally, in Section~\ref{sec:examples}, we analyze examples of certain bimolecular autocatalytic systems called \emph{autocatalytic recombination networks} where the dynamical systems generated by relative populations can be shown to be persistent/permanent using results from reaction network theory. 


\section{Reaction networks}\label{sec:reaction_networks}

Here, we recall some basic terminology from the theory of reaction networks~\cite{gunawardena2003chemical,feinberg1989necessary,feinberg1987chemical,feinberg1988chemical,feinberg1972chemical}. A \emph{reaction network} can be represented as a directed graph $\GG=(V,E)$, where $V\subset\mathbb{R}^n$ and $E$ are the sets of vertices and edges respectively. Such graphs have also been called \emph{Euclidean embedded graphs} (abbreviated as E-graphs)~\cite{craciun2015toric,craciun2019polynomial,craciun2019endotactic}. In what follows, we refer to the edges in $\GG$ as reactions. If $(\s,\s')\in E$, then we write $\s\to \s'\in E$, where $\s$ is the \emph{source} vertex and $\s'$ is the \emph{target} vertex. The \textit{stoichiometric subspace} $S$ of a network is the vector space $S=\displaystyle\textrm{span}\{\s'-\s \mid \s\to \s'\in E \}$. For $\x_0\in\mathbb{R}^n_{>0}$, the  compatibility class of $\x_0$ is the affine subspace $(\x_0+ S)$, and the positive compatibility class of $\x_0$ is the polyhedron $(\x_0+ S)\cap\mathbb{R}^n_{> 0}$.\\

\noindent Let $\GG=(V,E)$ be a reaction network and let $V_S$ denote the set of source vertices. Then 
\begin{enumerate}[(i)]
\item $\GG$ is \textit{reversible} if $\s\to \s'\in E$ implies $\s'\to \s\in E$;
\item $\GG$ is \textit{weakly reversible} if each reaction is part of a cycle;
\item $\GG$ is \textit{endotactic}~\cite{craciun2013persistence} if for every $\w\in\mathbb{R}^n$ and $\s\to \s'\in E$ with $\w\cdot (\s'-\s)<0$, there exists $\tilde{\s}\to\tilde{\s}^{'}\in E$ such that $\w\cdot (\tilde{\s}^{'}-\tilde{\s})>0$ and $\w\cdot \tilde{\s} < \w\cdot \s$;
\item $\GG$ is \textit{strongly endotactic}~\cite{gopalkrishnan2014geometric} if for every $\w\in\mathbb{R}^n$ and $\s\to \s'\in E$ with $\w\cdot (\s'-\s)<0$, there exists $\tilde{\s}\to\tilde{\s}^{'}\in E$ such that $\w\cdot(\tilde{\s}^{'}-\tilde{\s})>0$, $\w\cdot \tilde{\s} < \w\cdot \s$ and $\w\cdot \tilde{\s} \leq \w\cdot \hat{\s}$ for all $\hat{\s}\in V_S$; and
\item a set  $L\subseteq V$ is a \emph{linkage class}~\cite{feinberg1987chemical} if $L$ is a maximal connected component of $\GG$.
\end{enumerate}
Every strongly endotactic reaction network is endotactic. A weakly reversible network is endotactic~\cite{craciun2013persistence}. Further, a weakly reversible reaction network with a single linkage class is strongly endotactic~\cite{gopalkrishnan2014geometric}. 
Proposition~\ref{test:parallel_sweep}, or the \emph{parallel sweep test}, provides a geometric test for strong endotacticity~\cite{gopalkrishnan2014geometric,craciun2013persistence}.

\begin{proposition}\label{test:parallel_sweep}
Consider a reaction network $\GG$ and its stoichiometric subspace $S$. For every vector $\w \not\in S^\perp$, let $H $ be the hyperplane perpendicular to $\w$ that contains a source vertex $\s_1$ such that for every other source vertex $\s_2$, we have $(\s_2-\s_1)\cdot \w \geq 0$. If for every reaction $\s \to \s'$ with $s \in H$,
we have $(\s'-\s)\cdot \w \geq 0$, and there exists a reaction $\s_0 \to \s'_0$ with $\s_0 \in H$ such that $(\s'_0-\s_0)\cdot \w > 0$, then the reaction network is said to have \textbf{passed} the parallel sweep test and is strongly endotactic. Else, it is not strongly endotactic.

\end{proposition}

\begin{proposition}\label{prop:endotactic_face_reaction}
Let $\GG$ be a reaction network such that all the vertices of $\GG$ are contained in the convex hull of its source vertices. Then $\GG$ is strongly endotactic if and only if for every proper face of the convex hull of the source vertices, there exists a reaction of $\GG$ with source vertex on this face and target vertex that does not belong to this face.
\end{proposition}

\begin{proof}
Let $S$ denote the stoichiometric subspace. Let $\w\in\mathbb{R}^n$ be such that $\w \not\in S^\perp$. Let $H $ be the hyperplane perpendicular to $\w$ that contains a source vertex $\s_1$ such that for every other source vertex $\s_2$, we have $(\s_2-\s_1)\cdot \w \geq 0$. Therefore, the intersection of $H$ with the convex hull of source vertices of $\GG$ is a proper face of the convex hull. Let us call this face $f$. 

($\Rightarrow$) First assume that $\GG$ is strongly endotactic, so it passes the parallel sweep test given in Proposition~\ref{test:parallel_sweep}.
In particular, for every reaction $\s \to \s'$ whose source vertex lies on $f$, we have $(\s'-\s)\cdot \w \geq 0$, and there exists a reaction $\s_0 \to \s'_0$ with $\s_0 \in f$ such that $(\s'_0-\s_0)\cdot \w > 0$. Since the vertices of $\GG$ are contained in the convex hull of its source vertices, the reaction $\s_0 \to \s'_0$ lies in the convex hull; the source vertex $\s_0$ lies on $f$ and the target vertex does not belong to $f$.

($\Leftarrow$) Let us assume that there exists a reaction $\s_0 \to \s'_0$ lying in the convex hull such that the source vertex lies on $f$ and the target vertex does not belong to $f$. In particular, this implies that $(\s'_0-\s_0)\cdot \w > 0$. We will show that $\GG$ is strongly endotactic by showing that $\GG$ passes the parallel sweep test given in~\ref{test:parallel_sweep}. Since all the vertices of $\GG$ are contained in the convex hull of its source vertices, for every reaction $\s \to \s'$ whose source vertex lies on $f$, we have $(\s'-\s)\cdot \w \geq 0$. This combined with the fact that $(\s'_0-\s_0)\cdot \w > 0$ shows that $\GG$ passes the parallel sweep test and is hence strongly endotactic.

\end{proof}

\begin{corollary}\label{prop:not_endotactic} 
Let $\GG$ be a reaction network such that all the vertices of $\GG$ are contained in the convex hull of its source vertices. If $\GG$ is not strongly endotactic, then there exists a proper face of the convex hull of the source vertices of $\GG$ such that every reaction with source on this face has target on this face.  
\end{corollary}

\begin{proof}
This follows from Proposition~\ref{prop:endotactic_face_reaction}.
\end{proof}

If we assume that the vertices of $\GG$ have non-negative integer components, then under  \emph{mass-action kinetics}~\cite{craciun2009toric,craciun2013persistence,craciun2019endotactic}, $\GG$ generates a dynamical system on $\mathbb{R}^n_{\geq 0}$ which can be expressed as
\begin{eqnarray}\label{eq:mass_action}
\frac{d\x}{dt} = \displaystyle\sum_{\s\to\s'\in E} k_{\s\to\s'}\x^{\s}(\s'-\s),
\end{eqnarray}
where $\x^{\s}={x_1}^{s_1}{x_2}^{s_2}\cdots{x_n}^{s_n}$ and $k_{\s\to \s'}>0$ is the rate constant corresponding to the reaction $\s\to \s'$. We will denote the dynamical system generated by mass-action kinetics as $\GG_{\bk}=(V,E,{\bk})$, where ${\bk} =  (k_{\s\to\s'})_{\s\to\s'\in E}$.

\begin{remark}\label{rem:test_mass-action}
A polynomial dynamical system consisting of equations of the form $\frac{dx_i}{dt}=f_i(\x)$, where $i=1,2,...,n$ and $\x=(x_1,x_2,...,x_n)^T$ is given by mass-action kinetics if and only if $x_i$ divides every negative monomial in $f_i(\x)$ for all $i\in \{1,2,...,n\}$~\cite{feinberg2019foundations,hars1981inverse}.
\end{remark}
%


\begin{definition}\label{defn:dynamical_equivalence}
Two mass-action systems $\GG_{\bk}=(V,E,{\bk})$ and $\tilde{G}_{\tilde{\bk}}=(\tilde{V},\tilde{E},\tilde{\bk})$ are \emph{dynamically equivalent} if they generate the same dynamical system~(\ref{eq:mass_action}), i.e.,
\begin{eqnarray}
\displaystyle\sum_{\s\to\s'\in E} k_{\s\to\s'}\x^{\s}(\s'-\s) = \displaystyle\sum_{\tilde{\s}\to\tilde{\s}'\in \tilde{E}} \tilde{k}_{\tilde{\s}\to\tilde{\s}'}\x^{\tilde{\s}}(\tilde{\s}'-\tilde{\s}). 
\end{eqnarray} 
%
Equivalently~\cite{craciun2020efficient}, two mass-action systems are dynamically equivalent if for every $\s_0 \in V \cup \tilde{V}$, 
\begin{eqnarray}
\displaystyle\sum_{\s_0\to\s'\in E} k_{\s_0\to\s'}(\s'-\s_0) = \displaystyle\sum_{\s_0\to\tilde{\s}'\in \tilde{E}} \tilde{k}_{\s_0\to\tilde{\s}'}(\tilde{\s}'-\s_0). 
\end{eqnarray}
\end{definition}

A dynamical system is called \emph{autonomous} if it can be written in the form $\frac{d\x(t)}{dt}=\f(\x(t))$. Generically, the rate coefficients in~(\ref{eq:mass_action}) can depend on time, giving rise to a non-autonomous system 
\begin{eqnarray}\label{eq:non_autonomous}
\frac{d\x}{dt} = \displaystyle\sum_{\s\to\s'\in E} k_{\s\to\s'}(t)\x^{\s}(\s'-\s).
\end{eqnarray}
If there exists an $\epsilon>0$ such that $\epsilon\leq k_{\s\to\s'}(t)\leq \frac{1}{\epsilon}$ for every rate constant $k_{\s\to\s'}(t)$ in~(\ref{eq:non_autonomous}), then the dynamical system is called a \emph{variable-$k$ mass-action system}. Note that every autonomous mass-action system is a variable-$k$ mass-action system. We now define some important dynamical properties.

\begin{definition}\label{defn:persistence}
%
A dynamical system given by~(\ref{eq:non_autonomous}) is said to be \emph{persistent} if a solution $\x(t)$ of~(\ref{eq:non_autonomous}) with any initial condition $\x(0)\in\mathbb{R}^n_{>0}$ satisfies $\displaystyle\liminf_{t\rightarrow T}\x_i(t)>0$ for all $i=1,2,..,n$, where $T\in(0,\infty]$ is the maximum time for which $\x(t)$ is defined. 
\end{definition}

\begin{definition}\label{defn:permanence}
Let $C$ be a positive compatibility class. A dynamical system given by~(\ref{eq:non_autonomous}) is said to be \emph{permanent}  on $C$ if there exists a compact set $K\subseteq C$ such that for every solution $\x(t)$ of~(\ref{eq:non_autonomous}) with initial condition $\x(0)\in C$, we have $\x(t)\in K$ for all $t$ sufficiently large. A dynamical system given by~(\ref{eq:non_autonomous}) is said to be permanent if it is permanent on any positive compatibility class. 
\end{definition}
It is easy to see that a permanent system is persistent. The above dynamical properties are related to important open problems in reaction network theory. The \emph{Persistence Conjecture} states that weakly reversible mass-action   systems are persistent. This conjecture has been generalized in~\cite{craciun2013persistence} to the \emph{Extended Permanence Conjecture}, which states that variable-$k$ endotactic mass-action  systems are permanent. These conjectures are intimately related to the more familiar \emph{Global Attractor Conjecture}~\cite{craciun2009toric}, which says that there is exactly one globally attracting steady state in every positive compatibility class for \emph{complex balanced  systems}~\cite{horn1972general}. 

\section{Autocatalytic systems}\label{sec:autocatalysis}


One of the goals of this paper is to make a connection between a certain type of networks called \emph{autocatalytic}~\cite{kauffman1986autocatalytic,hordijk2004detecting,hordijk2011required}, and dynamical properties like persistence and permanence using the framework of reaction network theory. Autocatalytic networks have been studied in the context of the origin of life~\cite{kauffman1996home,kauffman1993origins,hordijk2010autocatalytic}. An important class of autocatalytic networks called \emph{hypercycles}, which consist of a cyclic connection of molecules capable of self-replicating themselves by undergoing mutual catalysis, were believed to be responsible for the origin of life. {\em The classical n-dimensional hypercycle refers to the following network: $X_i + X_{i+1}\rightarrow X_i + 2X_{i+1}$ for $1\leq i\leq n$, where $X_{n+1}=X_1$(in the cyclic sense).} The dynamics of hypercycles has been a topic of interest since the $1980$'s where properties like permanence and the existence of a globally attracting fixed point were established for special cases 
of the hypercycle~\cite{schuster1979dynamical,hofbauer1988theory,eigen2012hypercycle,hofbauer1991stable}.  

The concentration of species in autocatalytic systems can become unbounded in finite time. One approach for solving the dynamical equations corresponding to autocatalytic systems is by the addition of a \emph{dilution flux}, which keeps the total concentration of species constant. It turns out that addition of a dilution flux is equivalent to analyzing the dynamics of relative populations of species in such networks. 
In general, it is \emph{not} true that the relative populations of species in an autonomous dynamical system are themselves solutions of an autonomous dynamical system. However, under certain assumptions on the original system, we show that the model of relative populations are solutions of an autonomous polynomial dynamical system up to time re-scaling. Theorem~\ref{thm:dynamics_relative_concentrations} illustrates this point. Further, we use this to show that the reaction networks corresponding to the relative populations of bimolecular autocatalytic systems have a certain form in Theorem~\ref{thm:relative_autocatalysis}. 

\begin{theorem}\label{thm:dynamics_relative_concentrations}
Consider an autonomous dynamical system $\mathcal{A}$ given by $\frac{d\x}{dt}=\f(\x)$, where $\f(\x)=(f_1(\x),f_2(\x),...,f_n(\x))$, where each $f_i(\x)$ is a homogeneous polynomial of degree $d$. Let $x_i(t)$ denote the solution corresponding to the $i^{th}$ component. Let $x_T(t)=\sum_{i=1}^n x_i(t)$ denote the total concentration. Then there exists an autonomous polynomial dynamical system $\tilde{\mathcal{A}}$ such that for any solution $\x(t)$ of $\mathcal{A}$, the function
    \[\tilde{\x}(t)=\frac{\x(t)}{x_T(t)}\]
defined for all time $t$ such that $0 < x_T(t) <\infty$, is, up to time-rescaling, a solution of $\tilde{\mathcal{A}}$.
Moreover, if $\mathcal{A}$ is a mass-action system, then $\tilde{\mathcal{A}}$ is also a mass-action system. In addition, if $\mathcal{A}$ is a variable-$k$ mass-action system, then $\tilde{\mathcal{A}}$ is also a variable-$k$ mass-action system.
\end{theorem}  

\begin{proof}
The proof proceeds by construction of $\tilde A$. Let $\tilde{\x} = {\x} / x_T$, then 
\begin{eqnarray}\label{eq:dynamical_system_relative}
\frac{d\tilde{\x}}{dt} ={x_T^{-2}}\left(\frac{d\x}{dt}x_T  - \x\displaystyle\sum_{i=1}^n \frac{dx_i}{dt}\right) =  {x_T^{-2}}\left(\f(\x)x_T  - \x\displaystyle\sum_{i=1}^n f_i(\x)\right).
\end{eqnarray}
Since $f_i(\x)$ is a homogeneous polynomial of degree $d$ and $\x=x_T\tilde{\x}$, we have $\f(\x)=x_T^d \f(\tilde{\x})$. Plugging this into Equation (\ref{eq:dynamical_system_relative}) gives
\begin{eqnarray}\label{eq:without_rescaling_system_relative}
\frac{d\tilde{\x}}{dt}=\left(\f(\tilde{\x})-\tilde{\x}\displaystyle\sum_{i=1}^n f_i(\tilde{\x})\right)x_T^{d-1},
\end{eqnarray}
which is a time-rescaled version of 
\begin{eqnarray}\label{eq:after_rescaling_system_relative}
\frac{d\tilde{\x}}{dt}= \f(\tilde{\x})-\tilde{\x}\displaystyle\sum_{i=1}^n f_i(\tilde{\x}).
\end{eqnarray}
Equations~(\ref{eq:without_rescaling_system_relative}) and (\ref{eq:after_rescaling_system_relative}) have the same set of trajectories and hence $\tilde{\x}$ is (after time-rescaling) a solution of $\tilde{\mathcal{A}}$ given by $\frac{d\tilde{\x}}{dt}=\tilde{\f}(\tilde{\x})$, where $\tilde{\f}(\tilde{\x})=\f(\tilde{\x})-\tilde{\x}\displaystyle\sum_{i=1}^n f_i(\tilde{\x})$.\\

If the system $\mathcal{A}$ was mass-action, by Remark~\ref{rem:test_mass-action} it follows that $\tilde{\mathcal{A}}$ is also mass-action. \\


Now consider the case when the system $\mathcal{A}$ is a variable-$k$ mass-action system, so that the right-hand side $\f$ is a system of homogeneous polynomials of degree $d$ with time-dependent coefficients, say
\begin{align*} 
\f(\x,t) =  \sum_{\s\to\s'\in E} k_{\s\to\s'}(t) \x^{\s} (\s' - \s). 
\end{align*}

%
\noindent In this case, the analogues of Equations~(\ref{eq:without_rescaling_system_relative}) and~(\ref{eq:after_rescaling_system_relative}) are
\begin{eqnarray}\label{eq:variable_without_rescaling_system_relative}
\frac{d\tilde{\x}}{dt}=\left(\f(\tilde{\x},t)-\tilde{\x}\displaystyle\sum_{i=1}^n f_i(\tilde{\x},t)\right)x_T^{d-1},
\end{eqnarray}
and
\begin{eqnarray}\label{eq:variable_after_rescaling_system_relative}
\frac{d\tilde{\x}}{dt}= \f(\tilde{\x},t)-\tilde{\x}\displaystyle\sum_{i=1}^n f_i(\tilde{\x},t)
\end{eqnarray}
respectively, where $x_T = \sum_i x_i$ and $x_i$ a solution of $\mathcal{A}$. 
Construct an augmented dynamical system to Equation~\ref{eq:variable_without_rescaling_system_relative} given by
\begin{eqnarray}\label{eq:augmented_system}
\begin{split}
\frac{d\tilde{\x}}{dt} &=\left(\f(\tilde{\x},t)-\tilde{\x}\displaystyle\sum_{i=1}^n f_i(\tilde{\x},t)\right)x_T^{d-1}(t) \\
\frac{d\tilde{y}}{dt} &= 1.       
\end{split}
\end{eqnarray}
If we assume that the initial condition satisfies $y(t_0) = t_0$ then Equation~\ref{eq:augmented_system} can also be expressed as a autonomous dynamical system in the following way
\begin{eqnarray}\label{eq:autonomous_augmented_system}
\begin{split}
\frac{d\tilde{\x}}{dt} &=\left(\f(\tilde{\x},\tilde{y})-\tilde{\x}\displaystyle\sum_{i=1}^n f_i(\tilde{\x},\tilde{y})\right)x_T^{d-1}(\tilde y) \\
\frac{d\tilde{y}}{dt} &= 1.       
\end{split}
\end{eqnarray}
\noindent For initial condition $(\tilde \x(t_0),\tilde y(t_0)) := (\tilde \x_0,\tilde y_0=\tilde t_0) \in  \mathbb{R}_{>0}^n\times \mathbb{R}_{>0}$, let $(\tilde{\x}_0(t),\tilde{y}_0(t))$ be the unique solution to~(\ref{eq:autonomous_augmented_system}). Let $\Gamma$ denote the trajectory curve corresponding to this solution. Consider the following system
\begin{align}\label{eq:new_system}
\begin{split}
\frac{d\tilde \x^*}{dt} &= \f(\tilde\x^*, \tilde y^*)-\tilde\x^*\sum_{i=1}^n f_i(\tilde\x^*, \tilde y^*) \\
\frac{d\tilde y^*}{dt} &= \frac{1}{x_T^{d-1}(\tilde y^*)}.
\end{split}
\end{align}
Note that the solution of system~(\ref{eq:new_system}) with the same initial conditions $(\tilde \x_0,\tilde y_0=\tilde t_0)$ also lies on $\Gamma$. If we now define $\tilde k(t) = k(\tilde y^*(t))$ then the solution of the dynamical system given by Equation~\ref{eq:new_system} restricted to the components of $\tilde\x^*$, with initial condition $(\tilde \x_0,\tilde y_0=\tilde t_0)$, is the same as the solution of the equation (\ref{eq:variable_after_rescaling_system_relative}) where we {\em replace} $\f(\x,t)$ by 
\begin{align*} 
\tilde\f(\x,t) = \sum_{\s\to\s'\in E} \tilde k_{\s\to\s'}(t) \x^{\s} (\s' - \s).
\end{align*}

\end{proof}

\begin{remark}
Using the fact $\displaystyle\sum_{i=1}^n \tilde{x}_i=1$, the dynamical system~(\ref{eq:after_rescaling_system_relative}) can also be written as
\begin{eqnarray}\label{eq:homogeneous_relative_concentrations}
\frac{d\tilde{\x}}{dt}= \f(\tilde{\x})\left(\displaystyle\sum_{i=1}^n \tilde{x}_i\right) - \tilde{\x}\displaystyle\sum_{i=1}^n f_i(\tilde{\x}),
\end{eqnarray}
whose right-hand side consists of homogeneous polynomials.
\end{remark}

From here on, we will consider dynamical systems governed by mass-action kinetics.

\begin{definition}\label{def:bimolecular_autocatalytic_systems}
A reaction network $\GG$ is said to be a {\em bimolecular autocatalytic system} if every reaction in $\GG$ is of the form $X_i + X_j \xrightarrow[]{} X_i + X_j + X_l$ where $i,j,l\in\{1,2,...,n\}$.
\end{definition}

The dynamics (after time-rescaling) generated by the relative population variables of a bimolecular autocatalytic system is generated by another reaction network under mass-action kinetics. For example, consider the reaction network $\GG=\{X_1 + X_2 \rightarrow X_1 + X_2 + X_3\}$. The network corresponding to the relative concentrations of $\GG$ is given by $\tilde{\GG}=\{2X_1 + X_2 \rightarrow X_1 + X_2 + X_3, X_1 + 2X_2 \rightarrow X_1 + X_2 + X_3\}$. The next theorem makes this precise.

\begin{theorem}\label{thm:relative_autocatalysis}
Consider a bimolecular autocatalytic system $\mathcal{M}$ with concentration variables \\ $x_1,x_2,...,x_n$ corresponding to species $X_1,X_2,...,X_n$. Let $\tilde{x_i}(t) = x_i(t)/x_T(t)$ where $1\leq i\leq n$ denote the relative population variables and $x_T(t)$ the total concentration. Then $\tilde{x}_1(t),\tilde{x}_2(t),...,\tilde{x}_n(t)$ are (after time-rescaling) solutions of a dynamical system generated by mass-action kinetics, consisting of reactions of the form $X_ p + X_i + X_j \xrightarrow[]{\kappa} X_i + X_j +  X_l$ for each reaction $X_i + X_j \xrightarrow[]{\kappa} X_i + X_j + X_l$ in the bimolecular autocatalytic system $\mathcal{M}$, where $i$, $j$, $l$, $p\in\{1,2,...,n\}$ such that $p\neq l$. 
\end{theorem}

\begin{proof}
We will consider a single reaction in the bimolecular autocatalytic system $\mathcal{M}$. By combining the right-hand sides of the other reactions, we will get our desired result. Consider the reaction $X_i + X_j \xrightarrow[]{\kappa} X_i + X_j + X_l$ where $i$, $j$, $l\in\{1,2,...,n\}$. We have the following differential equations corresponding to the species $X_i$, $X_j$, $X_l$, and every species $X_m$ that does not appear in the above reaction:
\begin{eqnarray}\label{eq:absolute_concentration}
\begin{split}
\frac{dx_i}{dt}  &= \kappa x_ix_j\delta_{li}\\
\frac{dx_j}{dt}  &= \kappa x_ix_j\delta_{jl} \\
\frac{dx_l}{dt}  &= \kappa x_ix_j \\
\frac{dx_m}{dt}  &= 0.
\end{split}
\end{eqnarray}
Some of the equations in~(\ref{eq:absolute_concentration}) may be redundant, e.g., when $i=l$ or $j=l$. According to Equation~(\ref{eq:homogeneous_relative_concentrations}), the dynamical system for the relative populations variables is 
\begin{eqnarray}\label{eq:relative_concetration_specific}
\begin{split}
\frac{d\tilde{x}_i}{dt}  & =  \kappa\tilde{x}_i\tilde{x}_j\delta_{li}\left(\displaystyle\sum_{r=1}^n \tilde{x}_r\right)-\kappa\tilde{x}^2_i\tilde{x}_j\\
\frac{d\tilde{x}_j}{dt}  &= \kappa\tilde{x}_i\tilde{x}_j\delta_{jl}\left(\displaystyle\sum_{r=1}^n \tilde{x}_r\right)  - \kappa\tilde{x}_i\tilde{x}^2_j\\
\frac{d\tilde{x}_l}{dt} & = \kappa\tilde{x}_i\tilde{x}_j\left(\displaystyle\sum_{r=1}^n \tilde{x}_r\right) - \kappa\tilde{x}_i\tilde{x}_j\tilde{x}_l\\
\frac{d\tilde{x}_m}{dt} & = - \kappa\tilde{x}_i\tilde{x}_j\tilde{x}_m.
\end{split}
\end{eqnarray}
We now show that dynamical system (\ref{eq:relative_concetration_specific}) can be generated by the set of reactions
\begin{eqnarray}\label{eq:imp_reaction}
\{X_ p + X_i + X_j \xrightarrow[]{\kappa} X_i + X_j + X_l \mid p\neq l\}.
\end{eqnarray}
There are four cases to consider:
\begin{enumerate}
\item $i=j=l$: The set of reactions of the form~(\ref{eq:imp_reaction}) contribute $\kappa\tilde{x}^2_i\left(\displaystyle\sum_{\substack{1\leq r\leq n \\ r\neq i}} \tilde{x}_r\right)$ to $\frac{d\tilde{x}_i}{dt}$ and $-\kappa\tilde{x}_i\tilde{x}_j\tilde{x}_m$ to $\frac{d\tilde{x}_m}{dt}$. 
\item $i=l$, $j\neq l$: The set of reactions of the form~(\ref{eq:imp_reaction}) contribute $\kappa\tilde{x}_i\tilde{x}_j\left(\displaystyle\sum_{\substack{1\leq r\leq n \\ r\neq i}} \tilde{x}_r\right)$ to $\frac{d\tilde{x}_i}{dt}$, $-\kappa\tilde{x}_i\tilde{x}^2_j$ to $\frac{d\tilde{x}_j}{dt}$, and $- \kappa\tilde{x}_i\tilde{x}_j\tilde{x}_m$ to $\frac{d\tilde{x}_m}{dt}$.
\item $i \neq l$, $j=l$: Similar to the case above, the set of reactions of the form~(\ref{eq:imp_reaction}) contribute $-\kappa\tilde{x}^2_i\tilde{x}_j$ to $\frac{d\tilde{x}_i}{dt}$, $\kappa\tilde{x}_i\tilde{x}_j\left(\displaystyle\sum_{\substack{1\leq r\leq n \\ r\neq j}} \tilde{x}_r\right)$ to $\frac{d\tilde{x}_j}{dt}$, and $- \kappa\tilde{x}_i\tilde{x}_j\tilde{x}_m$ to $\frac{d\tilde{x}_m}{dt}$.
\item $l\neq i$, $j\neq l$: The set of reactions of the form~(\ref{eq:imp_reaction}) contribute $-\kappa\tilde{x}^2_i\tilde{x}_j$ to $\frac{d\tilde{x}_i}{dt}$, $-\kappa\tilde{x}_i\tilde{x}^2_j$ to $\frac{d\tilde{x}_j}{dt}$, and $- \kappa\tilde{x}_i\tilde{x}_j\tilde{x}_m$ to $\frac{d\tilde{x}_m}{dt}$.
\end{enumerate}
Therefore, the mass-action system generated by~(\ref{eq:imp_reaction}) coincides with~(\ref{eq:relative_concetration_specific}).
\end{proof}

\section{Permanence and global stability of autocatalytic recombination networks}\label{sec:examples}

In what follows, we show that there exist very general bimolecular systems called autocatalytic recombination networks for which we can use Theorem~\ref{thm:relative_autocatalysis} to prove permanence of the dynamical systems generated by the relative populations in these networks. 

\emph{Genetic recombination} is a phenomenon that is widely believed to be responsible for variation among species~\cite{rieger2012glossary,alberts2002molecular}. It involves exchange of genetic material between molecules of DNA to produce a new molecule, which inherits certain properties of its parent molecules. Most familiar examples of recombination take place during prophase I of meiosis. In particular, for a single crossover recombination, two DNA sequences of equal length exchange genetic material to give a third sequence of the same length which has the \emph{prefix} of one of the sequences and the \emph{suffix} of the other sequence. Figure~\ref{fig:recombination} illustrates this point.

\begin{figure*}[h!]
\centering
\includegraphics[scale=0.45]{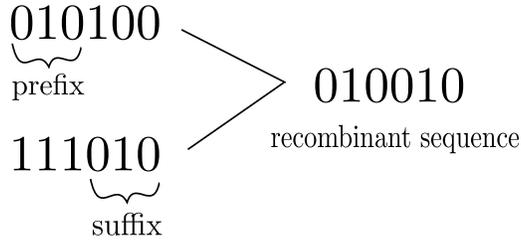}
\caption{Recombinant sequence produced from the prefix of one sequence and the suffix of another.}\label{fig:recombination}
\end{figure*} 
For our purpose, we analyze recombination reaction networks in the context of autocatalysis. The dynamical systems that we analyze bear striking resemblance to the fertility equations in~\cite{hofbauer1998evolutionary}. One can also view them as cyclic versions of the \emph{catalytic network equation} in~\cite{stadler1993random}. In particular, we analyze bimolecular autocatalytic networks (see Definition~\ref{def:bimolecular_autocatalytic_systems}), whose reactions are of the type $X_i + X_j \rightarrow X_i + X_j + X_k$. In this case, $X_k$ is the species formed by an autocatalytic reaction involving the exchange of genetic information between $X_i$ and $X_j$. In what follows, we shall show that the dynamical systems generated by relative populations of some of these reaction networks are permanent. 

In the next few examples, we consider autocatalytic recombinant networks involving the same reactant species. This is akin to \emph{homologous} genetic recombination~\cite{sung2006mechanism,vasquez2001manipulating} which involves the exchange of genetic material in the form of nucleotide sequences between two similar DNA strands. Homologous recombination plays an important role in repairing DNA strands that may be damaged due to chemicals and radiation. In addition, homologous recombination is used in \emph{gene targeting}~\cite{capecchi2005gene,capecchi1989altering,gerlai2016gene}, whereby certain genetic traits are introduced in a target organism.

Many of the theorems we prove use the geometry of a \emph{regular simplex}. Recall that a regular $n$-simplex is a regular $n$-polytope formed using the convex hull of its $n+1$ vertices. From here on, we shall refer to a regular $n$-simplex simply as a $n$-simplex.

Before we move on specific examples, we will define the notion of an autocatalytic recombination network with repeated species.

\begin{definition}
Consider a reaction network $\mathcal{G}$ with species $X_1,X_2,...,X_n$. Then $\mathcal{G}$ is said to be an \emph{autocatalytic recombination network of dimension $n$ with repeated reactant species} if it consists of reactions of the form $2X_i \rightarrow 3X_i$ and $2X_i \rightarrow 2X_i + X_{i+1}$ for $i=1,2,..,n$, where $X_{n+1}=X_1$(in the cyclic sense). 
\end{definition}

\subsection{\textbf{Autocatalytic recombination network of dimension three with repeated reactant species}} 

We consider an autocatalytic recombination network of dimension three given by $\GG^{\rm rep}_{\rm recomb(3)}$ in Table~\ref{same_species}. Here the reactant of each reaction is $2X_i$, where $1\leq i\leq 3$. By Theorem~\ref{thm:relative_autocatalysis},  the network $\tilde{\GG}^{\rm rep}_{\rm recomb(3)}$ given in Table~\ref{same_species} generates dynamics corresponding to the relative populations in $\GG^{\rm rep}_{\rm recomb(3)}$. The network $\tilde{\GG}^{\rm rep}_{\rm recomb(3)}$ is also depicted in Figure~\ref{fig:modified_hypercycle}.(a).

\begin{table}[h!]
\caption{Recombinant network of dimension three with repeated reactant species}
\centering
\begin{tabular}{|c|c|}
\hline
\rule{0pt}{20pt} $\GG^{\rm rep}_{\rm recomb(3)}$ & $\tilde{\GG}^{\rm rep}_{\rm recomb(3)}$  \\ [2ex]
\hline
$2X_1  \xrightarrow[]{k_1} 3X_1 $ & $2X_1 + X_2 \xrightarrow{k_1} 3X_1  $  \\
															   & $2X_1 + X_3 \xrightarrow{k_1} 3X_1 $ \\[1ex]
																				  
$2X_1 \xrightarrow[]{k_2} 2X_1 + X_2 $& $3X_1 \xrightarrow{k_2} 2X_1 + X_2$  \\
																		  & $2X_1 + X_3 \xrightarrow{k_2} 2X_1 + X_2 $ \\[1ex]
																		           
$2X_2  \xrightarrow[]{k_3} 3X_2  $ & $2X_2 + X_3 \xrightarrow{k_3} 3X_2$  \\
																  & $2X_2 + X_1 \xrightarrow{k_3} 3X_2$\\[1ex]
																  
$2X_2  \xrightarrow[]{k_4} 2X_2 + X_3 $ &$3X_2 \xrightarrow{k_4} 2X_2 + X_3$ \\
																			& $2X_2 + X_1 \xrightarrow{k_4} 2X_2 + X_3$  \\[1ex]
																              														  
$2X_3  \xrightarrow[]{k_5} 3X_3 $ & $2X_3 + X_1 \xrightarrow{k_5} 3X_3$  \\
																 & $2X_3 + X_2 \xrightarrow{k_5} 3X_3$\\[1ex]							
																              
$2X_3  \xrightarrow[]{k_6} 2X_3 + X_1 $ & $3X_3 \xrightarrow{k_6} 2X_3 + X_1$ \\
 																			& $2X_3 + X_2 \xrightarrow{k_6} 2X_3 + X_1$  \\[1ex]																															
																           																																      	\hline
\end{tabular}
\label{same_species}
\end{table}

\begin{theorem}
Any variable-$k$ dynamical system generated by $\tilde{\GG}^{\rm rep}_{\rm recomb(3)}$ is permanent.
\end{theorem}

\begin{proof}
We will show that $\tilde{\GG}^{\rm rep}_{\rm recomb(3)}$ is strongly endotactic. It will then follow  from~\cite{gopalkrishnan2014geometric} that any variable-$k$ dynamical system generated by it is permanent. The convex hull formed by the source vertices of $\tilde{\GG}^{\rm rep}_{\rm recomb(3)}$ is a triangle as shown in Figure~\ref{fig:modified_hypercycle}.(b). In particular, the triangle contains all the vertices of $\tilde{\GG}^{\rm rep}_{\rm recomb(3)}$. By Proposition~\ref{prop:endotactic_face_reaction}, to show that $\tilde{\GG}^{\rm rep}_{\rm recomb(3)}$ is strongly endotactic, it suffices to show that for every proper face of the triangle, there exists a reaction with source on this face and target that does not belong to this face. One can check that this is the case from Figure~\ref{fig:modified_hypercycle}.(a). Therefore, $\tilde{\GG}^{\rm rep}_{\rm recomb(3)}$ is strongly endotactic. 
\end{proof}

\begin{figure*}[h!]
\centering
\includegraphics[scale=0.4]{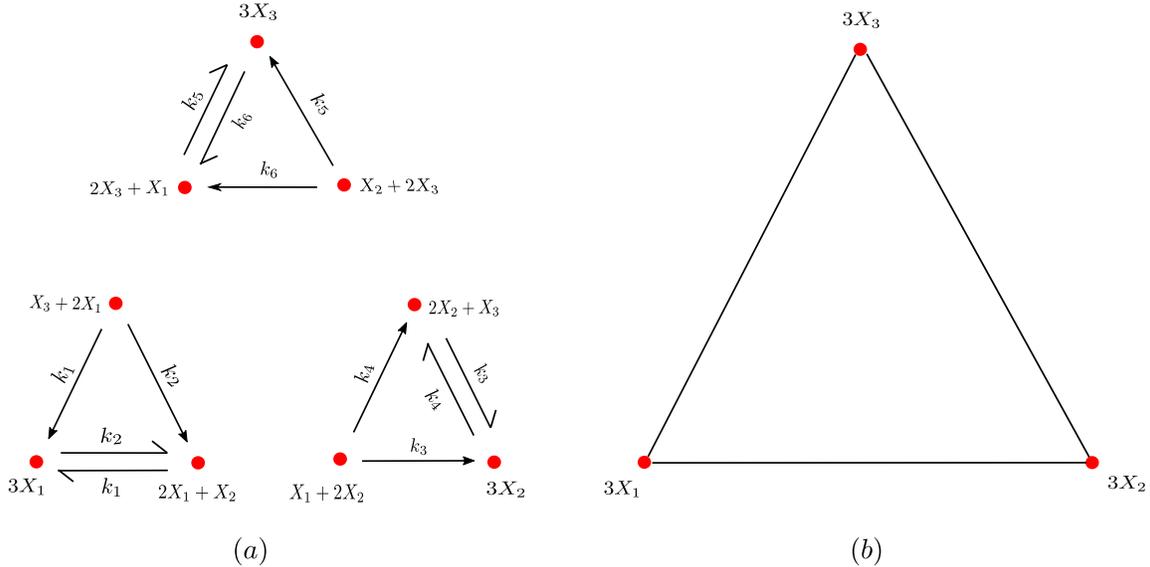}
\caption{(a) Reaction network that generates the dynamics of relative populations of the recombination network of dimension three with repeated reactant species given by $\GG^{\rm rep}_{\rm recomb(3)}$. (b) Convex hull of the source vertices of the network in Figure~\ref{fig:modified_hypercycle}.(a).}
\label{fig:modified_hypercycle}
\end{figure*}

\subsection{\textbf{Autocatalytic recombination network of dimension four with repeated reactant species}}

Here, we consider an autocatalytic recombination network of dimension four given by $\GG^{\rm rep}_{\rm recomb(4)}$ in Table~\ref{same_species_4}, where the reactant of each reaction is $2X_i$, where $1\leq i\leq 4$. 
By Theorem~\ref{thm:relative_autocatalysis}, the network for relative populations is given by $\tilde{\GG}^{\rm rep}_{\rm recomb(4)}$ in Table~\ref{same_species_4} and Figure~\ref{fig:modified_higher_hypercycle}.

\begin{table}[h!]
\caption{Recombinant network of dimension four with repeated reactant species}
\centering
\begin{tabular}{|c|c|}
\hline
\rule{0pt}{20pt} ${\GG}^{\rm rep}_{\rm recomb(4)}$  & $\tilde{\GG}^{\rm rep}_{\rm recomb(4)}$  \\ [2ex]
\hline
$2X_1  \xrightarrow[]{k_1} 3X_1 $ & $2X_1 + X_2 \xrightarrow{k_1} 3X_1  $  \\
															   & $2X_1 + X_3 \xrightarrow{k_1} 3X_1 $ \\
															   & $2X_1 + X_4 \xrightarrow{k_1} 3X_1 $ \\[1ex]	
															   	  
$2X_1 \xrightarrow[]{k_2} 2X_1 + X_2 $& $3X_1 \xrightarrow{k_2} 2X_1 + X_2$  \\
																		  & $2X_1 + X_3 \xrightarrow{k_2} 2X_1 + X_2 $ \\
																		   & $2X_1 + X_4 \xrightarrow{k_2} 2X_1 + X_2 $ \\[1ex]    
																		     
$2X_2  \xrightarrow[]{k_3} 3X_2  $ & $2X_2 + X_3 \xrightarrow{k_3} 3X_2$  \\
																  & $2X_2 + X_1 \xrightarrow{k_3} 3X_2$\\
																  & $2X_2 + X_4 \xrightarrow{k_3} 3X_2$\\[1ex]
																  
$2X_2  \xrightarrow[]{k_4} 2X_2 + X_3 $ &$3X_2 \xrightarrow{k_4} 2X_2 + X_3$ \\
																			& $2X_2 + X_1 \xrightarrow{k_4} 2X_2 + X_3$  \\
																			& $2X_2 + X_4 \xrightarrow{k_4} 2X_2 + X_3$\\[1ex]
																              														  
$2X_3  \xrightarrow[]{k_5} 3X_3 $ & $2X_3 + X_1 \xrightarrow{k_5} 3X_3$  \\
																 & $2X_3 + X_2 \xrightarrow{k_5} 3X_3$\\
																 & $2X_3 + X_4 \xrightarrow{k_5} 3X_3$	 \\[1ex]			
																              
$2X_3  \xrightarrow[]{k_6} 2X_3 + X_4 $ & $3X_3 \xrightarrow{k_6} 2X_3 + X_4$ \\
 																			& $2X_3 + X_2 \xrightarrow{k_6} 2X_3 + X_4$  \\
 																			& $2X_3 + X_4 \xrightarrow{k_6} 2X_3 + X_4$ \\[1ex]		
 																			
$2X_4  \xrightarrow[]{k_7} 3 X_4 $ & $2X_4 + X_1\xrightarrow{k_7} 3X_4$ \\
 																 &  $2X_4 + X_2 \xrightarrow{k_7} 3X_4$  \\
 																 &  $2X_4 + X_3 \xrightarrow{k_7} 3X_4$ \\[1ex]

$2X_4  \xrightarrow[]{k_8} 2X_4 + X_1 $ & $3X_4 \xrightarrow{k_8} 2X_4 + X_1$ \\
 																			& $2X_4 + X_2 \xrightarrow{k_8} 2X_4 + X_1$  \\
 																			& $2X_4 + X_3 \xrightarrow{k_8} 2X_4 + X_1$ \\[1ex]																	  														
\hline
\end{tabular}
\label{same_species_4}
\end{table}

\begin{theorem}\label{thm:4_strongly_endotactic}
Any variable-$k$ dynamical system generated by $\tilde{\GG}^{\rm rep}_{\rm recomb(4)}$ is permanent.
\end{theorem}

\begin{proof}
We will show that $\tilde{\GG}^{\rm rep}_{\rm recomb(4)}$ is strongly endotactic. It will then follow from~\cite{gopalkrishnan2014geometric} that any variable-$k$ dynamical system generated by it is permanent.
The convex hull formed by the source vertices of $\tilde{\GG}^{\rm rep}_{\rm recomb(4)}$ is the tetrahedron shown in Figure~\ref{fig:modified_higher_hypercycle}. In particular, the tetrahedron contains all the vertices of $\tilde{\GG}^{\rm rep}_{\rm recomb(4)}$. By Proposition~\ref{prop:endotactic_face_reaction}, to show that $\tilde{\GG}^{\rm rep}_{\rm recomb(4)}$ is strongly endotactic, it suffices to show that for every proper face of the tetrahedron, there exists a reaction with source on this face and target that does not belong to this face. One can check that this is the case from Figure~\ref{fig:modified_higher_hypercycle}. Therefore, $\tilde{\GG}^{\rm rep}_{\rm recomb(4)}$ is strongly endotactic. 
\end{proof}

\begin{figure*}[h!]
\centering
\includegraphics[scale=0.57]{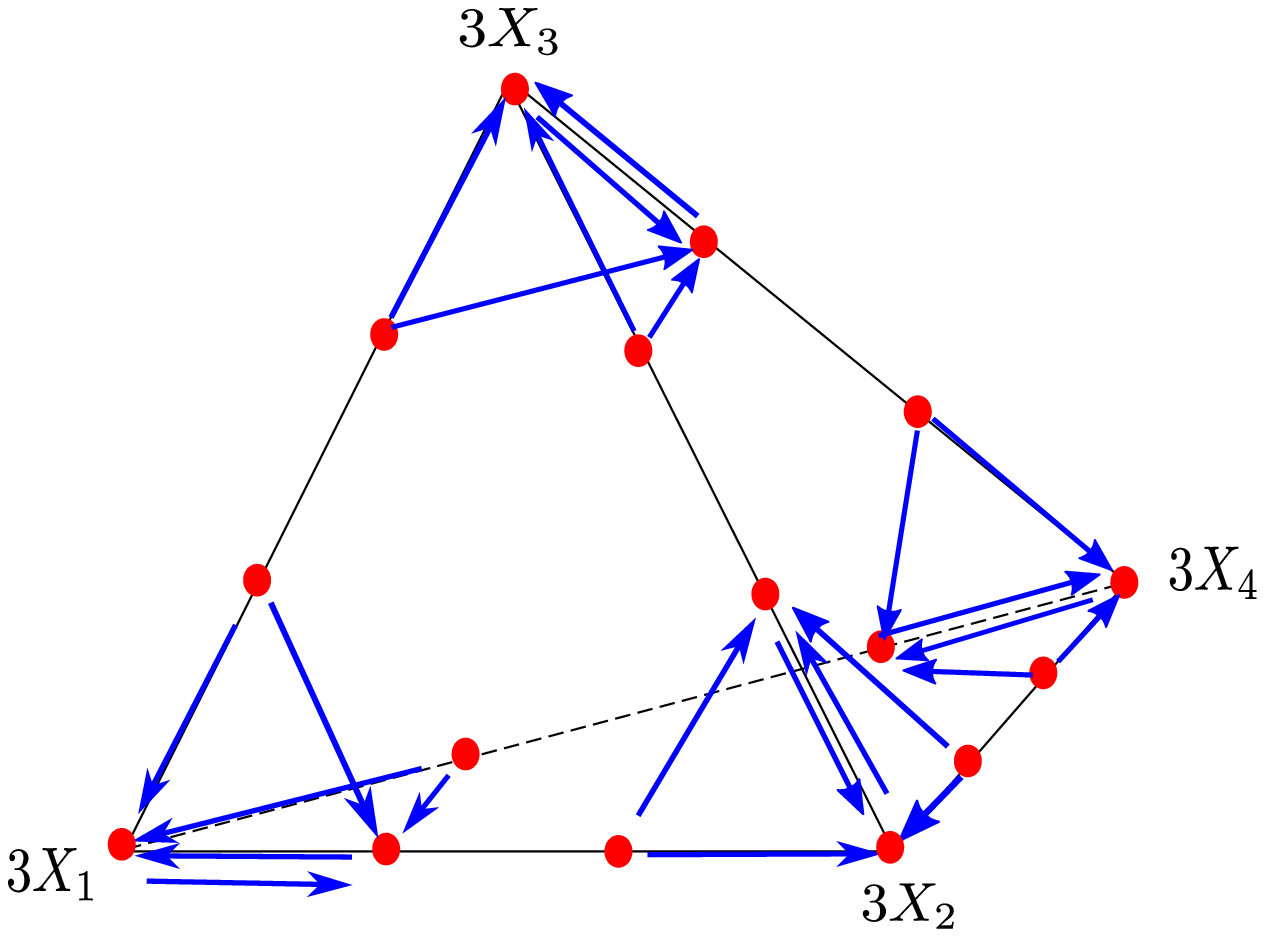}
\caption{Reaction network $\tilde{\GG}^{\rm rep}_{\rm recomb(4)}$ of Table~\ref{same_species_4}.}\label{fig:modified_higher_hypercycle}
\end{figure*} 

\subsection{\textbf{Autocatalytic recombination network of dimension n with repeated reactant species}}

We claim that the variable-$k$ dynamical system generated by the relative population network corresponding to the autocatalytic recombination network with repeated reactant species in higher dimensions, the analog of $\tilde{\GG}^{\rm rep}_{\rm recomb(3)}$, is also permanent. Consider ${\GG}^{\rm rep}_{\rm recomb(n)}$ consisting of the following reactions
\begin{align*}
2X_1 &\rightarrow 3X_1  \\
2X_1 &\rightarrow 2X_1 + X_2 \\
& \quad\vdots \\
2X_{n-1} &\rightarrow 3X_{n-1} \\
2X_{n-1} &\rightarrow 2X_{n-1} + X_n \\
2X_{n} &\rightarrow 3X_{n} \\
2X_{n-1} &\rightarrow 2X_{n} + X_1 .\\
\end{align*}
Let $\tilde{\GG}^{\rm rep}_{\rm recomb(n)}$ be constructed as in Theorem~\ref{thm:relative_autocatalysis}. Then we have the following:

\begin{theorem}\label{cor:recombination_dim_n}
Any variable-$k$ dynamical system generated by $\tilde{\GG}^{\rm rep}_{\rm recomb(n)}$ is permanent.
\end{theorem}

\begin{proof}
We will show that $\tilde{\GG}^{\rm rep}_{\rm recomb(n)}$ is strongly endotactic. It will then follow from~\cite{gopalkrishnan2014geometric} that any variable-$k$ dynamical system generated by it is permanent. Consider the convex hull formed by the source vertices in $\tilde{\GG}^{\rm rep}_{\rm recomb(n)}$. Each of the vertices $3X_1,3X_2,...,3X_n$ are sources for some reaction in $\tilde{\GG}^{\rm rep}_{\rm recomb(n)}$ and are extremal points of the convex hull. Therefore this convex hull is a $(n-1)$-simplex. For contradiction, assume that $\tilde{\GG}^{\rm rep}_{\rm recomb(n)}$ is not strongly endotactic. Then by Corollary~\ref{prop:not_endotactic}, there exists a proper face of this simplex such that for every reaction with source on this face has target on this face. Let $F=(3X_{i_1},3X_{i_2},...,3X_{i_k})\subset (3X_1,3X_2,...,3X_n)$ be this face. Note that $\tilde{\GG}^{\rm rep}_{\rm recomb(n)}$ contains reactions of the form $3X_j\rightarrow 2X_j  + X_{j+1}$ for every $3X_j\in F$. Since $\tilde{\GG}^{\rm rep}_{\rm recomb(n)}$ was assumed to be not strongly endotactic, we get that $2X_j + X_{j+1}\in F$. Extending the line segment joining $3X_j$ and $2X_j + X_{j+1}$, we get that $3X_{j+1}\in F$. Continuing this way, we conclude $F=(3X_1,3X_2,...,3X_n)$, contradicting the fact that $F$ is a proper face of this simplex. Therefore, $\tilde{\GG}^{\rm rep}_{\rm recomb(n)}$ is strongly endotactic. 
\end{proof}
Theorem~\ref{cor:recombination_dim_n} can be generalized to a larger family of reaction networks that have certain properties associated with an object called the \emph{production graph}. The next proposition illustrates this point. We first define the notion of a production graph. 

\begin{definition}\label{defn:production_graph}
Given a reaction network $\GG=(V,E)$, the \emph{production graph} $\mathcal{P}(\GG)$ is a graph whose vertices are given by the species in $\GG$ such that there is a directed edge from species $X_i$ to species $X_j$ in $\mathcal{P}(\GG)$ if there exists an edge $\by\to\by'\in E$ such that $\supp(\by)=X_i$ and $X_j\in \supp(\by')$.
\end{definition}

\begin{proposition}\label{prop:strongly_connected}
Consider a bimolecular autocatalytic system $\GG=(V,E)$ consisting of reactions of the form $2X_i \rightarrow 2X_i + X_j$ for  $i=1,2,...,n$ and $j\neq i$ such that $\mathcal{P}(\GG)$ is strongly connected. Let $\tilde{\GG}$ denote the reaction network corresponding to the relative populations of $\GG$. Then any variable-$k$ dynamical system generated by $\tilde{\GG}$ is permanent.
\end{proposition}

\begin{proof}
We will show that $\tilde{\GG}$ is strongly endotactic. It will then follow from~\cite{gopalkrishnan2014geometric} that any variable-$k$ dynamical system generated by it is permanent. Since $P(\GG)$ is strongly connected, there exists a reaction $2X_i \rightarrow 2X_i + X_j$ for each $i=1,2,...,n$. By Theorem~\ref{thm:relative_autocatalysis} there exists reactions in $\tilde{\GG}$ which are of the form $3X_j\rightarrow 2X_j  + X_k$ such that $k\neq j$ for every species $X_j$. Therefore the convex hull of the source vertices of $\tilde{\GG}$ is a simplex with extremal points $(3X_1,3X_2,...,3X_n)$. For contradiction, assume that $\tilde{\GG}^{\rm rep}_{\rm recomb(n)}$ is not strongly endotactic. Then by Corollary~\ref{prop:not_endotactic}, there exists a proper face of this simplex such that for every reaction with source on this face has target on this face. Let $F=(3X_{i_1},3X_{i_2},...,3X_{i_k})\subset (3X_1,3X_2,...,3X_n)$ be this face. Since $\mathcal{P}(\GG)$ is strongly connected, one can argue as in the proof of Theorem~\ref{cor:recombination_dim_n} to show that $F=(3X_1,3X_2,...,3X_n)$, contradicting the fact that $F$ is a proper face of the simplex. Therefore, the network $\tilde{\GG}$ is strongly endotactic.
\end{proof}

The relative population network $\tilde{\GG}$ in Proposition~\ref{prop:strongly_connected} can be enlarged to a new network (under some constraints) so that the dynamical system generated by this new network is still permanent. The next corollary illustrates this point.   

\begin{proposition}\label{prop:larger_endotactic}
Consider reaction networks $\tilde{G}_1$, $\tilde{G}_2$ such that the following hold
\begin{enumerate}
\item $\tilde{G}_1$ is strongly endotactic.
\item $\tilde{G}_1$ is a subnetwork of $\tilde{G}_2$.
\item The vertices of $\tilde{G}_2$ lie in the convex hull of the source vertices of $\tilde{G}_1$. 
\end{enumerate}
Then any variable-$k$ dynamical system generated by $\tilde{G}_2$ is permanent.
\end{proposition}

\begin{proof}
We will show that $\tilde{G}_2$ is strongly endotactic. It will then follow from~\cite{gopalkrishnan2014geometric} that any variable-$k$ dynamical system generated by it is permanent. Since $\tilde{G}_1 \subseteq\tilde{G}_2$ and the vertices of $\tilde{G}_2$ lie in the convex hull of the source vertices of $\tilde{G}_1$, the convex hull of the source vertices of $\tilde{G}_1$ is the same as the convex hull of the source vertices of $\tilde{G}_2$. Let $\w\in\mathbb{R}^n$ be such that $\w \not\in S^\perp$. Let $H $ be the hyperplane perpendicular to $\w$ that contains a source vertex $\s_1$ such that for every other source vertex $\s_2$, we have $(\s_2-\s_1)\cdot \w \geq 0$. Therefore, the intersection of $H$ with the convex hull of source vertices of $\GG$ is a proper face of the convex hull. Since $\tilde{G}_1$ is strongly endotactic, by the parallel sweep Proposition~\ref{test:parallel_sweep}, there exists a reaction in $\tilde{G}_1$ with source on this face such that it points inside this convex hull. Since $\tilde{G}_1 \subseteq\tilde{G}_2$, this reaction is also contained in $\tilde{G}_2$. This implies that $\tilde{G}_2$ is strongly endotactic.
\end{proof}

\begin{corollary}
Suppose $G_1$ is a bimolecular autocatalytic network with repeated species and $P(G_1)$ is strong connected. Suppose $G_2$ is another bimolecular autocatalytic network such that $G_1\subseteq G_2$. Let $\tilde{G}_1$ and $\tilde{G}_2$ be the networks corresponding to the relative populations of $G_1$ and $G_2$ respectively. Then any variable-$k$ dynamical system generated by $\tilde{G}_2$ is permanent.
\end{corollary}

\begin{proof}
Since $P(G_1)$ strongly connected and $G_1$ is a bimolecular autocatalytic network with repeated species, the convex hull of source vertices of $\tilde{G}_1$ has corners $3X_1, 3X_2,\ldots, 3X_n$ (as outlined in the proof of Proposition~\ref{prop:strongly_connected}). The fact that $G_2$ is a bimolecular autocatalytic network implies that the vertices of $\tilde{G}_2$ are contained in convex hull formed by $(3X_1, 3X_2,\ldots, 3X_n)$, i.e., the convex hull formed by the sources of $\tilde{G}_1$. Note that since $G_1\subseteq G_2$, we have $\tilde{G}_1\subseteq \tilde{G}_2$. In addition, $\tilde{G}_1$ is strongly endotactic by Proposition~\ref{prop:strongly_connected}. The result now follows from Proposition~\ref{prop:larger_endotactic}.
\end{proof}






In the next few examples, we consider autocatalytic recombinant networks involving different reactant species. This is similar in spirit to \emph{nonhomologous} genetic recombination, which involves the exchange of genetic material in the form of nucleotide sequences between two dissimilar DNA strands. Nonhomologous genetic recombination is used for repairing breaks in DNA strands. Below, we give a precise definition for such networks.

\begin{definition}
Consider a reaction network $\mathcal{G}$ with species $X_1,X_2,...,X_n$. Then $\mathcal{G}$ is said to be an \emph{autocatalytic recombination network of dimension $n$} if it consists of reactions of the form $X_i + X_{i+1} \rightarrow X_i + 2X_{i+1}$ and $X_i + X_{i+1} \rightarrow X_i + X_{i+1} + X_{i+2}$ for $i=1,2,..,n$, where $X_{n+1}=X_1$ and $X_{n+2}=X_2$(in the cyclic sense). 
\end{definition}

\subsection{\textbf{Autocatalytic recombination network of dimension three}} 

Consider the network $\GG_{\rm recomb(3)}$ (given in Table~\ref{mixed_hypercycle_3}). Most notably, species only interact with \emph{other} species. By Theorem~\ref{thm:relative_autocatalysis}, the dynamics of the relative populations of $\GG_{\rm recomb(3)}$ can be generated by the network $\tilde{\GG}_{\rm recomb(3)}$ shown in Table~\ref{mixed_hypercycle_3}. 

\begin{table}[h!]
\caption{Recombinant network of dimension three}
\centering
\begin{tabular}{|c|c|}
\hline
\rule{0pt}{20pt} $\GG_{\rm recomb(3)}$ & $\tilde{\GG}_{\rm recomb(3)}$  \\ [2ex]
\hline
$X_1 + X_2 \xrightarrow[]{k_1} X_1 + 2X_2$ & $2X_1 + X_2 \xrightarrow{k_1} X_1 + 2X_2 $  \\
																				  & $X_1 + X_2 + X_3 \xrightarrow{k_1} X_1 + 2X_2$ \\[1ex]
$X_2 + X_3\xrightarrow[]{k_2} X_2 + 2X_3$& $2X_2 + X_3 \xrightarrow{k_2} X_2 + 2X_3 $  \\
																		           & $X_1 + X_2 + X_3 \xrightarrow{k_2} X_2 + 2X_3 $ \\[1ex]
$X_3 + X_1 \xrightarrow[]{k_3} X_3 + 2X_1 $ & $2X_3 + X_1 \xrightarrow{k_3} X_3 + 2X_1$  \\
																					& $X_1 + X_2 + X_3 \xrightarrow{k_3}X_3 + 2X_1$\\[1ex]
$X_1 + X_2 \xrightarrow[]{k_4} X_1 + X_2 + X_3$ & $X_1 + 2X_2 \xrightarrow{k_4} X_1 + X_2 +X_3  $\\
																							& $2X_1 + X_2 \xrightarrow{k_4} X_1 + X_2 +X_3$\\[1ex]
$X_2 + X_3 \xrightarrow[]{k_5} X_1 + X_2 + X_3$ & $X_2 + 2X_3 \xrightarrow{k_5} X_1 + X_2 +X_3$ \\
 											   												& $2X_2 + X_3 \xrightarrow{k_5} X_1 + X_2 +X_3$\\[1ex]
$X_1 + X_3 \xrightarrow[]{k_6} X_1 + X_2 + X_3$ & $X_1 + 2X_3 \xrightarrow{k_6} X_1 + X_3 +X_3 $ \\ 
																 	                         & $2X_1 + X_3 \xrightarrow{k_6} X_1 + X_2 +X_3$ \\[1ex]
 \hline
\end{tabular}
\label{mixed_hypercycle_3}
\end{table}
The reaction network $\tilde{\GG}_{\rm recomb(3)}$ is illustrated in Figure~\ref{fig:non_lotka_volterra}.(a).

\begin{theorem}\label{cor:three_dim_recomb}
Any variable-$k$ dynamical system generated by $\tilde{\GG}_{\rm recomb(3)}$ is permanent.
\end{theorem}

\begin{proof}
Any dynamical system generated by the network in Figure~\ref{fig:non_lotka_volterra}.(a) can be obtained from the dynamical system generated by the network in Figure~\ref{fig:non_lotka_volterra}.(b) (which is weakly reversible and possesses a single linkage class) if we choose the rate constants as shown in Figure~\ref{fig:non_lotka_volterra}. It follows from~\cite{boros2020permanence} that any variable-$k$ dynamical system generated by $\tilde{\GG}_{\rm recomb(3)}$ is permanent .
\end{proof}

\begin{figure*}[h!]
\centering
\includegraphics[scale=0.4]{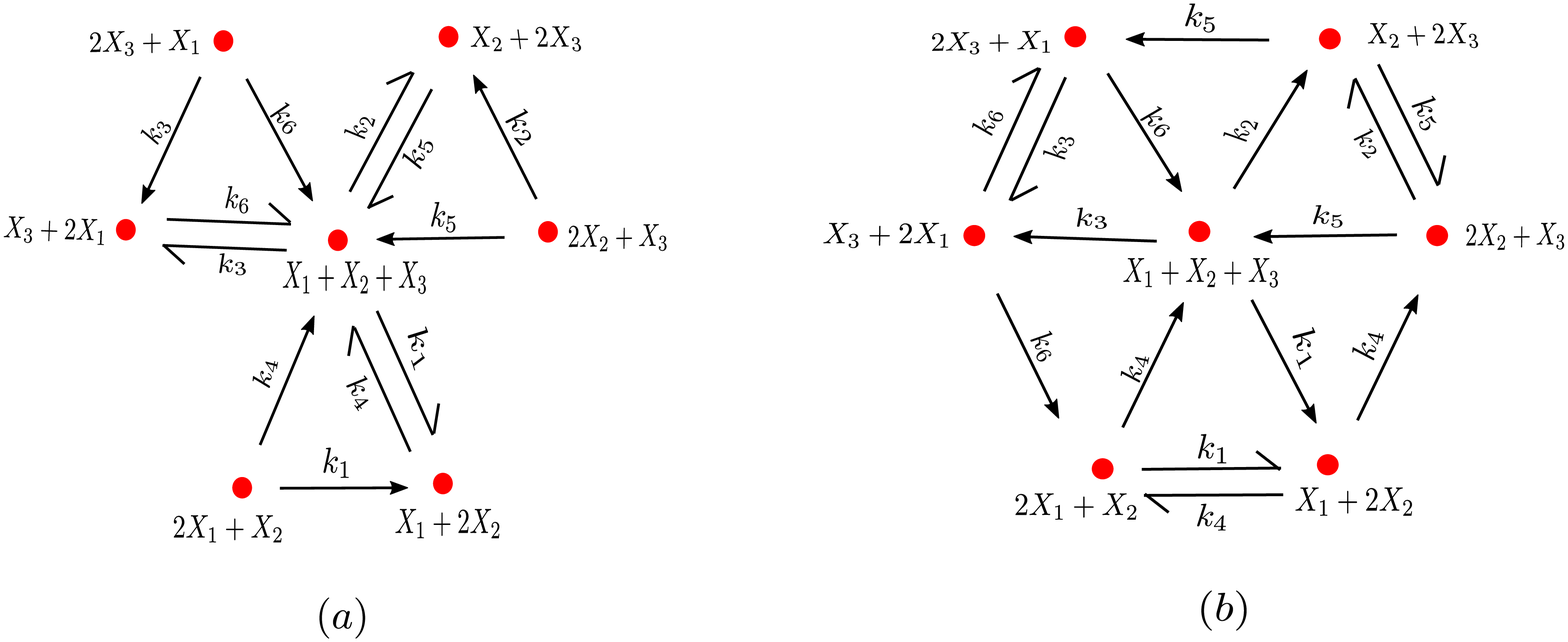}
\caption{(a) Reaction network that generates the dynamics of relative populations of the recombination network of dimension three given by $\GG_{\rm recomb(3)}$ given by $\tilde{\GG}_{\rm recomb(3)}$. (b) Weakly reversible reaction network with single linkage class that generates the same dynamics as that given by the network in Figure~\ref{fig:non_lotka_volterra}.(a), if we choose rate constants as shown above.}
\label{fig:non_lotka_volterra}
\end{figure*}

\subsection{\textbf{Autocatalytic recombination network of dimension four, five and six}}\label{sec:recomb}

In this subsection, we analyze autocatalytic recombination networks of dimension four, five and six. First, consider the  autocatalytic recombination network of dimension four given by $\GG_{\rm recomb(4)}$ as depicted in Table~\ref{mixed_hypercycle_4}. The corresponding network for relative populations is given by $\tilde{\GG}_{\rm recomb(4)}$ and is also given in Table~\ref{mixed_hypercycle_4}. In Figure~\ref{fig:higher_hypercycle}.(a), we illustrate the reaction network corresponding to the relative population of species for the following subset of reactions in $\GG_{\rm recomb(4)}$: reactions with sources $X_1 + X_2$ and $X_2 + X_3$.

\begin{table}[h!]
\caption{Recombinant network of dimension four}
\centering
\begin{tabular}{|c|c|}
 \hline
\rule{0pt}{20pt} $\GG_{\rm recomb(4)}$  & $\tilde{\GG}_{\rm recomb(4)}$  \\
 [2 ex]\hline
$X_1 + X_2 \xrightarrow[]{k_1} X_1 + 2X_2$ & $2X_1 + X_2 \xrightarrow{k_1} X_1 + 2X_2 $  \\
																				  & $X_1 + X_2 + X_3 \xrightarrow{k_1} X_1 + 2X_2$ \\
																				  & $X_1 + X_2 + X_4 \xrightarrow{k_1} X_1 + 2X_2$ \\[1ex]
																				  
$X_2 + X_3\xrightarrow[]{k_2} X_2 + 2X_3$& $2X_2 + X_3 \xrightarrow{k_2} X_2 + 2X_3 $  \\
																		           & $X_1 + X_2 + X_3 \xrightarrow{k_2} X_2 + 2X_3 $ \\
																		           & $X_2 + X_3 + X_4 \xrightarrow{k_2} X_2 + 2X_3 $ \\[1ex]
																		           
$X_3 + X_4 \xrightarrow[]{k_3} X_3 + 2X_4 $ & $2X_3 + X_4 \xrightarrow{k_3} X_3 + 2X_4$  \\
																					& $X_1 + X_3 + X_4 \xrightarrow{k_3}X_3 + 2X_4$\\
																					& $X_2 + X_3 + X_4 \xrightarrow{k_3}X_3 + 2X_4$\\[1ex]
																				
$X_4 + X_1 \xrightarrow[]{k_4} X_4 + 2X_1 $ & $2X_4 + X_1 \xrightarrow{k_4} X_4 + 2X_1$  \\
																					& $X_1 + X_2 + X_4 \xrightarrow{k_4} X_4 + 2X_1$\\																														& $X_1 + X_3 + X_4 \xrightarrow{k_4} X_4 + 2X_1$\\	[1ex]									
$X_1 + X_2 \xrightarrow[]{k_5} X_1 + X_2 + X_3$ & $X_1 + 2X_2 \xrightarrow{k_5} X_1 + X_2 +X_3  $\\
																							& $2X_1 + X_2 \xrightarrow{k_5} X_1 + X_2 +X_3$\\
																							& $X_1 + X_2 + X_4\xrightarrow{k_5} X_1 + X_2 +X_3$\\[1ex]
																							
$X_2 + X_3 \xrightarrow[]{k_6} X_2 + X_3 + X_4$ & $X_2 + 2X_3 \xrightarrow{k_6} X_2 + X_3 +X_4$ \\
 											   												& $2X_2 + X_3 \xrightarrow{k_6} X_2 + X_3 +X_4$\\
 											   									   			& $X_1 + X_2 + X_3 \xrightarrow{k_6} X_2 + X_3 +X_4$\\ 		[1ex]								   												
 											   												
$X_3 + X_4 \xrightarrow[]{k_7} X_3 + X_4 + X_1$ & $X_3 + 2X_4 \xrightarrow{k_7} X_3 + X_4 +X_1 $ \\ 
																 	                          & $2X_3 + X_4 \xrightarrow{k_7} X_3 + X_4 +X_1 $  \\
																 	                          & $X_2 + X_3 + X_4 \xrightarrow{k_7} X_3 + X_4 +X_1 $  \\[1ex]
																 	                          
$X_4 + X_1 \xrightarrow[]{k_8} X_4 + X_1 + X_2$ & $X_4 + 2X_1 \xrightarrow{k_8} X_4 + X_1 + X_2 $ \\ 
																 	                          & $2X_4 + X_1 \xrightarrow{k_8} X_4 + X_1 + X_2$  \\
																 	                          & $X_1 + X_3 + X_4 \xrightarrow{k_8} X_4 + X_1 + X_2$  \\[1ex]																 	                          
																 	                          
 \hline
\end{tabular}
\label{mixed_hypercycle_4}
\end{table}

\begin{figure*}[h!]
\centering
\includegraphics[scale=0.5]{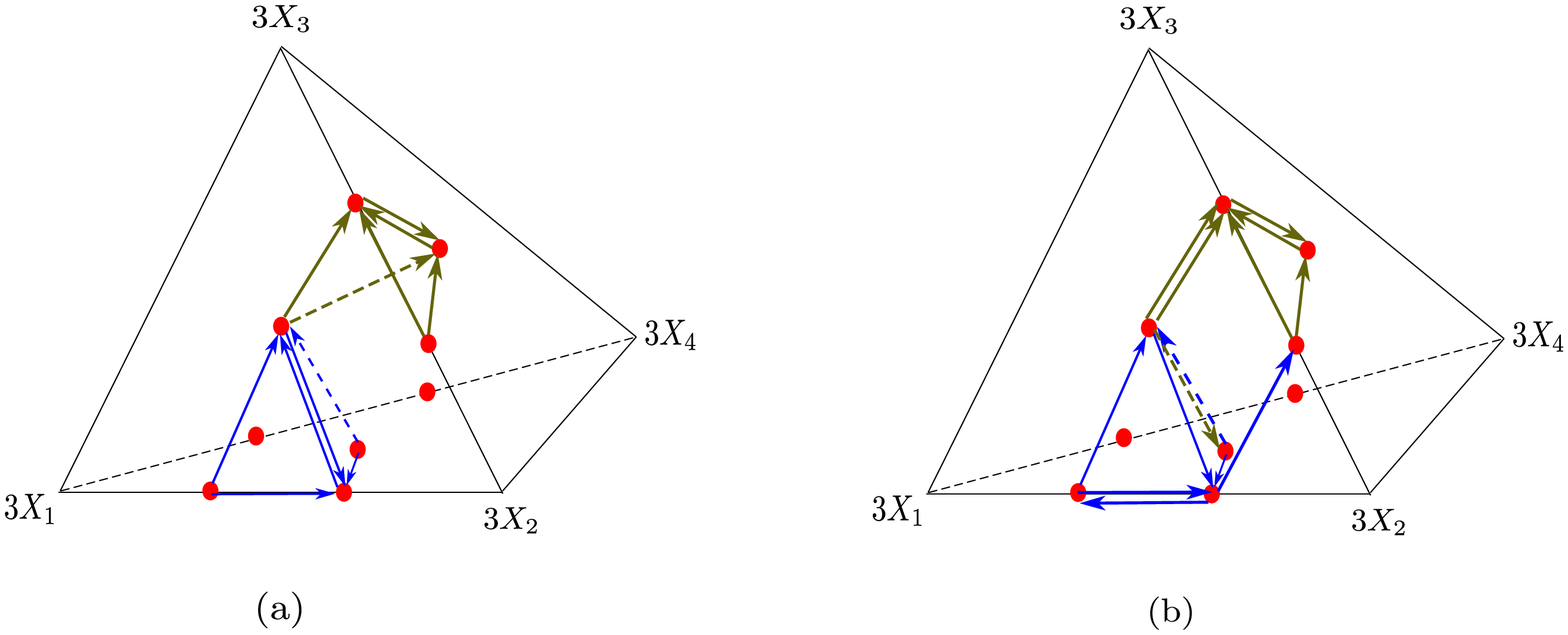}
\caption{(a) A network illustrating a subset of reactions in $\tilde{\GG}_{\rm recomb(4)}$. (b) A reaction network that is dynamically equivalent to the network in (a). The reaction $X_1+2X_2\rightarrow X_1 + X_2 + X_3$ (marked with blue in network (a)) is split as $X_1+2X_2\rightarrow 2X_1 + X_2$ and $X_1+2X_2\rightarrow X_3 + 2X_2$ (marked with blue  in network (b)). The reaction $X_1+ X_2 + X_3\rightarrow X_2 + X_3 + X_4$ (marked with green in network (a)) is split as $X_1+X_2+X_3\rightarrow X_2 + 2X_3$ and $X_1+ X_2 + X_3\rightarrow X_1 + X_2 + X_4$ (marked with green in network (b))}. 
\label{fig:higher_hypercycle}
\end{figure*} 

In what follows, we show that the reaction networks corresponding to the relative populations of autocatalytic recombinant networks of dimension four, five and six given by $\tilde{\GG}_{\rm recomb(4)},\tilde{\GG}_{\rm recomb(5)}$ and $\tilde{\GG}_{\rm recomb(6)}$ can be made dynamically equivalent to a weakly reversible reaction network with a single linkage class. As a consequence, any variable-$k$ dynamical system generated by these networks in permanent~\cite{boros2020permanence}.

\begin{proposition}\label{prop:G_4_WR}
The reaction network $\tilde{\GG}_{\rm recomb(4)}$ can be made dynamically equivalent to a weakly reversible reaction network with a single linkage class.
\end{proposition}

\begin{proof}
For the reactions $X_1 + X_2 \rightarrow X_1 + 2X_2$ and $X_1 + X_2 \rightarrow X_1 + X_2 + X_3$ in $\GG_{\rm recomb(4)}$, Theorem~\ref{thm:relative_autocatalysis} gives us the following subnetwork in $\tilde{\GG}_{\rm recomb(4)}$: 
    \begin{center}
    \begin{tikzpicture}
        \node (x2y) at (0,0) [left] {$X_1 + 2X_2$};
        \node (xyz) at (2,0) [right] {$X_1+X_2+X_3$};
        \node (xyw) at (1,1.5) {$X_1+X_2+X_4$};
        \node (2xy) at (1,-1.5) {$2X_1+X_2$};
        \draw [-{Stealth}, thick]  (xyw)--(x2y); 
        \draw [-{Stealth}, thick]  (xyw)--(xyz);
        \draw [-{Stealth}, thick]  (2xy)--(x2y); 
        \draw [-{Stealth}, thick]  (2xy)--(xyz);
        \draw [-{Stealth[left]}, thick, transform canvas={yshift=1pt}]  (x2y)--(xyz);
        \draw [-{Stealth[left]}, thick, transform canvas={yshift=-1pt}]  (xyz)--(x2y);
    \node at (5,0) [right] {};
    \end{tikzpicture} 
    \end{center}
whose geometric embedding is shown in Figure~\ref{fig:higher_hypercycle}.(a). To make this subnetwork dynamically equivalent to a weakly reversible reaction network, it therefore suffices to find appropriate reactions with targets $X_1+ X_2 + X_4$ and $2X_1 + X_2$. We can accomplish this using the following:
\begin{enumerate}[(i)]
\item \emph{Reaction with target $2X_1 + X_2$}: We split the reaction $X_1 + 2X_2 \rightarrow X_1 + X_2 + X_3$ into
\begin{eqnarray*}
\begin{split}
X_1 + 2X_2 &\rightarrow 2X_1 + X_2\ \rm{and} \\
X_1 + 2X_2 &\rightarrow 2X_2 + X_3,
\end{split}
\end{eqnarray*} 
as shown in Figure~\ref{fig:higher_hypercycle}.(b), since $(0,-1,1,0)^T = (1,-1,0,0)^T + (-1,0,1,0)^T$.
\item\label{seq:4} \emph{Reaction with target $X_1+ X_2 + X_4$}: This can done with the following sequence of reactions
\begin{eqnarray*}
\begin{split}
X_1+ X_2 + X_4 &\rightarrow X_1+ X_2 + X_3 \\
X_1+ X_2 + X_3 &\rightarrow X_2+ X_3 + X_4 \\
X_2+ X_3 + X_4 &\rightarrow X_3+ X_4 + X_1 \\
X_3+ X_4 + X_1 &\rightarrow X_1+ X_2 + X_4,
\end{split}
\end{eqnarray*}
which are known to exist in $\tilde{\GG}_{\rm recomb(4)}$ from Table~\ref{mixed_hypercycle_4}.
\end{enumerate}
Repeating this procedure for the remaining network, we obtain a weakly reversible reaction network. In addition, the sequence of reactions described in~\ref{seq:4} ensures that this network consists of a single linkage class. Therefore, the reaction network $\tilde{\GG}_{\rm recomb(4)}$ can be made dynamically equivalent to a weakly reversible reaction network with a single linkage class. 
\end{proof}

\begin{proposition}\label{prop:G_5_WR}
The reaction network $\tilde{\GG}_{\rm recomb(5)}$ can be made dynamically equivalent to a weakly reversible reaction network with a single linkage class.
\end{proposition}

\begin{proof}
For the reactions $X_1 + X_2 \rightarrow X_1 + 2X_2$ and $X_1 + X_2 \rightarrow X_1 + X_2 + X_3$ in $\GG_{\rm recomb(5)}$, Theorem~\ref{thm:relative_autocatalysis} gives us the following subnetwork in $\tilde{\GG}_{\rm recomb(5)}$: 
    
    \begin{center}
    \begin{tikzpicture}
        \node (l0) at (0,0) [left] {$X_1 + 2X_2$};
        \node (r0) at (2,0) [right] {$X_1+X_2+X_3$};
        \node (l1) at (0,1.5) [left] {$X_1+X_2+X_4$};
        \node (r1) at (2,1.5) [right] {$X_1+X_2+X_5$};
        \node (2xy) at (1,-1.5) {$2X_1+X_2$};
        \draw [-{Stealth}, thick]  ([xshift=10pt]l1.south)--(l0); 
        \draw [-{Stealth}, thick]  (l1)--(r0);
        \draw [-{Stealth}, thick]  (r1)--(l0); 
        \draw [-{Stealth}, thick]  (r1)--(r0);
        \draw [-{Stealth}, thick]  (2xy)--(l0); 
        \draw [-{Stealth}, thick]  (2xy)--(r0);
        \draw [-{Stealth[left]}, thick, transform canvas={yshift=1pt}]  (l0)--(r0);
        \draw [-{Stealth[left]}, thick, transform canvas={yshift=-1pt}]  (r0)--(l0);
    \node at (5,0) [right] {};
    \end{tikzpicture} 
    \end{center}
    
To make this subnetwork dynamically equivalent to weakly reversible single linkage class, it therefore suffices to find appropriate reactions with targets $X_1+ X_2 + X_4, 2X_1 + X_2$ and $X_1+ X_2 + X_5$. We can accomplish this using the following:
\begin{enumerate}[(i)]
\item \emph{Reaction with target $2X_1 + X_2$}: We split the reaction $X_1 + 2X_2 \rightarrow X_1 + X_2 + X_3$ into
\begin{eqnarray*}
\begin{split}
X_1 + 2X_2 &\rightarrow 2X_1 + X_2\ \rm{and} \\
X_1 + 2X_2 &\rightarrow 2X_2 + X_3,
\end{split}
\end{eqnarray*} 
since $(0,-1,1,0,0)^T = (1,-1,0,0,0)^T + (-1,0,1,0,0)^T$.
\item \emph{Reaction with target $X_1+ X_2 + X_4$}: We split the reaction $X_1+ X_2 + X_3 \rightarrow X_2+ X_3 + X_4$ into 
\begin{eqnarray*}
\begin{split}
X_1+ X_2 + X_3 &\rightarrow X_2+ 2X_3\ \rm{and} \\
X_1+ X_2 + X_3 &\rightarrow X_1 + X_2+ X_4,
\end{split}
\end{eqnarray*}
as shown in Figure~\ref{fig:higher_hypercycle}.(b), since $(-1,0,0,1,0)^T = (-1,0,1,0,0)^T + (0,0,-1,1,0)^T$.
\item\label{seq:5} \emph{Reaction with target $X_1+ X_2 + X_5$}: This can done with the following sequence of reactions
\begin{eqnarray*}
\begin{split}
X_1+ X_2 + X_5 &\rightarrow X_1+ X_2 + X_3\\
X_1+ X_2 + X_3 &\rightarrow X_2+ X_3 + X_4 \\
X_2+ X_3 + X_4 &\rightarrow X_3+ X_4 + X_5 \\
X_3+ X_4 + X_5 &\rightarrow X_4+ X_5 + X_1 \\
X_4+ X_5 + X_1 &\rightarrow X_1+ X_2 + X_5,
\end{split}
\end{eqnarray*}
which are known to exist in $\tilde{\GG}_{\rm recomb(5)}$.
\end{enumerate}
Repeating this procedure for the remaining network, we obtain a weakly reversible reaction network. In addition, the sequence of reactions described in~\ref{seq:5} ensures that this network consists of a single linkage class. Therefore, the reaction network $\tilde{\GG}_{\rm recomb(5)}$ can be made dynamically equivalent to a weakly reversible reaction network with a single linkage class. 

\end{proof}

\begin{proposition}\label{prop:G_6_WR}
The reaction network $\tilde{\GG}_{\rm recomb(6)}$ can be made dynamically equivalent to a weakly reversible reaction network with a single linkage class.
\end{proposition}

\begin{proof}
For the reactions $X_1 + X_2 \rightarrow X_1 + 2X_2$ and $X_1 + X_2 \rightarrow X_1 + X_2 + X_3$ in $\GG_{\rm recomb(6)}$, Theorem~\ref{thm:relative_autocatalysis} gives us the following subnetwork in $\tilde{\GG}_{\rm recomb(6)}$: 
    
    \begin{center}
    \begin{tikzpicture}
        \node (l0) at (0,0) [left] {$X_1 + 2X_2$};
        \node (r0) at (2,0) [right] {$X_1+X_2+X_3$};
        \node (l1) at (0,1.5) [left] {$X_1+X_2+X_4$};
        \node (r1) at (2,1.5) [right] {$X_1+X_2+X_5$};
        \node (2xy) at (0,-1.5) [left] {$2X_1+X_2$};
        \node (r2) at (2,-1.5) [right] {$X_1+X_2+X_6$};
        \draw [-{Stealth}, thick]  ([xshift=10pt]l1.south)--(l0); 
        \draw [-{Stealth}, thick]  (l1)--(r0);
        \draw [-{Stealth}, thick]  (r1)--(l0); 
        \draw [-{Stealth}, thick]  (r1)--(r0);
        \draw [-{Stealth}, thick]  (r2)--(l0); 
        \draw [-{Stealth}, thick]  (r2)--(r0);
        \draw [-{Stealth}, thick]  (2xy)--(l0); 
        \draw [-{Stealth}, thick]  (2xy)--(r0);
        \draw [-{Stealth[left]}, thick, transform canvas={yshift=1pt}]  (l0)--(r0);
        \draw [-{Stealth[left]}, thick, transform canvas={yshift=-1pt}]  (r0)--(l0);
    \node at (5,0) [right] {};
    \end{tikzpicture} 
    \end{center}
    
To make this subnetwork dynamically equivalent to weakly reversible single linkage class, it therefore suffices to find appropriate reactions with targets $X_1+ X_2 + X_4, 2X_1 + X_2, X_1+ X_2 + X_5$ and $X_1 + X_2 + X_6$. We can accomplish this using the following:
\begin{enumerate}[(i)]
\item \emph{Reaction with target $2X_1 + X_2$}: We split the reaction $X_1 + 2X_2 \rightarrow X_1 + X_2 + X_3$ into
\begin{eqnarray*}
\begin{split}
X_1 + 2X_2 &\rightarrow 2X_1 + X_2\ \rm{and} \\
X_1 + 2X_2 &\rightarrow 2X_2 + X_3,
\end{split}
\end{eqnarray*} 
since $(0,-1,1,0,0,0)^T = (1,-1,0,0,0,0)^T + (-1,0,1,0,0,0)^T$.
\item \emph{Reaction with target $X_1+ X_2 + X_4$}:  We split the reaction $X_1+ X_2 + X_3 \rightarrow X_2+ X_3 + X_4$ into 
\begin{eqnarray*}
\begin{split}
X_1+ X_2 + X_3 &\rightarrow X_2+ 2X_3\ \rm{and} \\
X_1+ X_2 + X_3 &\rightarrow X_1 + X_2+ X_4,
\end{split}
\end{eqnarray*}
since $(-1,0,0,1,0,0)^T = (-1,0,1,0,0,0)^T + (0,0,-1,1,0,0)^T$.
\item \emph{Reaction with target $X_1+ X_2 + X_5$}: We split the reaction $X_5+ X_6 + X_1 \rightarrow X_6+ X_1 + X_2$ into 
\begin{eqnarray*}
\begin{split}
X_5+ X_6 + X_1 &\rightarrow X_1+ X_2 + X_5\ \rm{and} \\
X_5+ X_6 + X_1 &\rightarrow X_1 + 2X_6,
\end{split}
\end{eqnarray*}
since $(0,1,-1,0,0,0)^T = (0,1,0,-1,0,-1)^T + (0,0,-1,1,0,1)^T$
\item\label{seq:6} \emph{Reaction with target $X_1+ X_2 + X_6$}: This can done with the following sequence of reactions:
\begin{eqnarray*}
X_1+ X_2 + X_6 &\rightarrow X_1+ X_2 + X_3 \\
X_1+ X_2 + X_3 &\rightarrow X_2+ X_3 + X_4 \\
X_2+ X_3 + X_4 &\rightarrow X_3+ X_4 + X_5 \\
X_3+ X_4 + X_5 &\rightarrow X_4+ X_5 + X_6 \\
X_4+ X_5 + X_6 &\rightarrow X_5+ X_6 + X_1 \\
X_5+ X_6 + X_1 &\rightarrow X_1+ X_2 + X_6,
\end{eqnarray*}
which are known to exist in $\tilde{\GG}_{\rm recomb(6)}$.
\end{enumerate}
Repeating this procedure for the remaining network, we obtain a weakly reversible reaction network. In addition, the sequence of reactions described in~\ref{seq:6} ensures that this network consists of a single linkage class. Therefore, the reaction network $\tilde{\GG}_{\rm recomb(6)}$ can be made dynamically equivalent to a weakly reversible reaction network with a single linkage class. 
\end{proof}

\begin{corollary}
Any variable-$k$ dynamical system generated by $\tilde{\GG}_{\rm recomb(4)},\tilde{\GG}_{\rm recomb(5)}$ and $\tilde{\GG}_{\rm recomb(6)}$ is permanent.
\end{corollary}

\begin{proof}
We have shown above that $\tilde{\GG}_{\rm recomb(4)},\tilde{\GG}_{\rm recomb(5)}$ and $\tilde{\GG}_{\rm recomb(6)}$ are dynamically equivalent to weakly reversible networks with single linkage class. It follows from~\cite{boros2020permanence} that any variable-$k$ dynamical system generated by them is permanent.
\end{proof}

There is a larger family of recombinant networks which can be shown to be strongly endotactic, as we show in Theorem~\ref{thm:property_X}. We first need a lemma about truncated $n$-simplexes. A truncated $n$-simplex is constructed by truncating (cutting off) each vertex at third of the length of the edge of the regular $n$-simplex. Figure~\ref{fig:truncated_n_simplex} shows a few examples of truncated simplices for $n=2$ and $n=3$. Note that the facet of a truncated $n$-simplex is either a face of a $(n-1)$-simplex or a face of a truncated $(n-1)$-simplex.

\begin{figure*}[h!]
\centering
\includegraphics[scale=0.5]{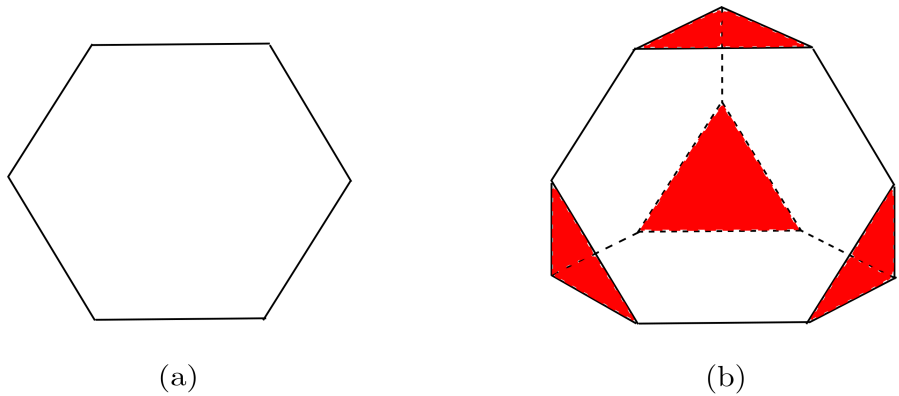}
\caption{(a) Truncated 2-simplex, which is a hexagon. The facets of the hexagon are truncated line segments. (b) Truncated 3-simplex. The facets of the truncated 3-simplex are hexagons and triangles (marked in red).}\label{fig:truncated_n_simplex}
\end{figure*} 
\begin{lemma}\label{lem:face_truncated_simplex}
For $n\geq 3$, any face of the truncated $n$-simplex is either a face of a $(n-1)$-simplex or an entire truncated $r$-simplex where $r\leq n-1$, but not both. 
\end{lemma}

\begin{proof}
By definition, a face of a truncated $n$-simplex is either a face of a $(n-1)$-simplex or a face of a truncated $(n-1)$-simplex. It therefore suffices to show that a face of a truncated $(n-1)$-simplex that is not a face of a $(n-1)$-simplex is an entire truncated $r$-simplex for some $r\leq n-1$. We show this by induction. Base case ($n=3$): Refer to Figure~\ref{fig:truncated_n_simplex}.(b). Here, the face of a truncated 3-simplex that is not a face of the 2-simplex (where the cut happened) is either an entire hexagon (which is a truncated 2-simplex) or a certain side of the hexagon (which is a truncated 1-simplex).

Let us assume that the property mentioned above is true for $m<n-1$. The induction hypothesis gives us that a face of a truncated $(n-2)$-simplex that is not a face of an $(n-2)$-simplex is an entire truncated $r$-simplex for some $r\leq n-2$. By definition, the face of a truncated $(n-1)$-simplex is either a face of a $(n-2)$-simplex or a face of a truncated $(n-2)$-simplex. Now note that a face of a truncated $(n-1)$-simplex that is not a face of the $(n-1)$-simplex cannot be a face of a $(n-2)$-simplex. This implies that it is a face of a truncated $(n-2)$-simplex that is not a face of an $(n-2)$-simplex. By the induction hypothesis, we get that this face is an entire truncated $r$-simplex for some $r\leq n-1$.
\end{proof}

\begin{theorem}\label{thm:property_X}
Consider an E-graph $\GG_n=(V,E)$ with $n$ species $X_1, X_2,..., X_n$ having the following property: for any mutually distinct indices $i,j\in\{1,2,...,n\}$, there exists a reaction $X_i+X_j \rightarrow X_i+X_j + X_k$ in $\GG_n$ with $k\notin\{i,j\}$, and there exists a way to use $X_i,X_j,X_k$ to make a fourth species, then a fifth species, and so on, until we get all the $n$ species. Let us denote by $\tilde{\GG_n}$ the reaction network corresponding to the relative populations of $\GG_n$. Then any variable-$k$ dynamical system generated by $\tilde{\GG_n}$ is permanent.
\end{theorem}

\begin{proof}
We will show that $\tilde{\GG}_n$ is strongly endotactic. It will then follow from~\cite{gopalkrishnan2014geometric} that any variable-$k$ dynamical system generated by it is permanent. Note that from the hypothesis of the theorem, we have that for all $i \neq j$ there exists reactions in $\GG_n$ that are of the form $X_i + X_j \rightarrow X_i + X_j + X_k$, with $k\neq\{i,j\}$. By Theorem~\ref{thm:relative_autocatalysis}, we get that $\tilde{\GG}_n$ has complexes of the form $2X_i + X_j$ or $X_i + 2X_j$ as sources, for all $i \neq j$, and these are the vertices of the truncated $(n-1)$-simplex. Therefore, the convex hull of the source vertices of $\tilde{\GG}_n$ is a truncated $(n-1)$-simplex. Since all the vertices of $\tilde{\GG}_n$ are contained in the convex hull of its source vertices, by Proposition~\ref{prop:endotactic_face_reaction}, $\tilde{\GG}_n$ is strongly endotactic if and only if for every face of the convex hull of the source vertices, there is a reaction with source on this face that points away from this face. Consider a face $f$ of the truncated $(n-1)$-simplex. By Lemma~\ref{lem:face_truncated_simplex}, we know that $f$ is either a face of a $(n-2)$-simplex or $f$ is an entire truncated $r$-simplex generated by some species $X_{i_1},...,X_{i_r}$ for some $r\leq n-2$. We have the following cases:
\begin{enumerate}
\item If $f$ is a face of a $(n-2)$-simplex: Take a point of the form $2X_i + X_j$ on $f$ and consider the following reaction in $\GG_n$: $X_i + X_j\rightarrow X_i + X_j + X_k$ where $k\notin\{i,j\}$. By Theorem~\ref{thm:relative_autocatalysis}, the network $\tilde{\GG}_n$ contains a reaction $2X_i + X_j\rightarrow X_i + X_j + X_k$. Note that $X_i + X_j + X_k$ does not belong to the $(n-2)$-simplex. Therefore, $X_i + X_j + X_k$ does not belong to $f$. 

\item If $f$ is an entire truncated $r$-simplex generated by some species $X_{i_1},...,X_{i_r}$: We now use the property that, for any distinct indices $i,j\in\{1,2,...,n\}$, there exists a reaction $X_i+X_j \rightarrow X_i+X_j + X_k$ in $\GG_n$ with $k\notin\{i,j\}$, and there exists a way to use $X_i,X_j,X_k$ to make a fourth species, then a fifth species,..., until we get all the $n$ species. Therefore, there exists a way to use some of the species $X_{i_1},...,X_{i_r}$ to obtain a species $X_{i_s}$ that is not on the face $f$, which gives us our desired reaction.

\end{enumerate}
Therefore, $\tilde{\GG}_n$ is strongly endotactic.
\end{proof}

\section{Discussion and future work}

Autocatalytic networks are ubiquitous in nature. The populations of species in autocatalytic networks have the potential to become unbounded in finite time. As a consequence, the key quantity in this context is relative populations. It is not clear before hand whether relative populations can be solutions of polynomial dynamical systems. Surprisingly, under time rescaling and certain assumptions on the original reaction network, we show that they become solutions of polynomial dynamical systems. Theorem~\ref{thm:dynamics_relative_concentrations} establishes this fact. In particular, in Theorem~\ref{thm:relative_autocatalysis}, we give explicit reaction networks that generate dynamics corresponding to the relative populations of bimolecular autocatalytic systems. 

In Section~\ref{sec:examples}, we study certain examples of bimolecular autocatalytic systems called \emph{autocatalytic recombination networks}. Using tools from reaction network theory, we show that variable-$k$ dynamical systems generated by the relative populations of certain autocatalytic recombination networks are permanent. A unifying feature of these networks is that the vertices corresponding to the relative population reaction networks are contained in the convex hull of their source vertices. Therefore, the variable-$k$ dynamical systems generated by various generalizations of these networks are permanent as long as the vertices stay inside this convex hull. It is notable that the network $\tilde{\GG}_{\rm recomb(n)}$ in Section~\ref{sec:recomb} can be made dynamically equivalent to a weakly reversible reaction network with single linkage class for $n\leq 6$. The method used to show this dynamical equivalence does not work for $n\geq 7$; nevertheless it is possible that these networks are strongly endotactic. Generalizing this to higher dimensions is one of the main challenges for future work. Another direction for future investigation would be to find conditions under which the results of ``gene targeting" via homologous recombination introduced in Section~\ref{sec:examples} will persist in the cell owing to our results showing persistence in such networks.

\bibliographystyle{amsplain}
\bibliography{Bibliography}

\end{document}